\documentclass[11pt]{article}
\usepackage{epigamath}

\usepackage[notext]{kpfonts}
\usepackage{baskervald}

\setpapertype{A4}

\usepackage[english]{babel}

\usepackage[dvips]{graphicx}     % Use this instead with dvips

\title{\vspace{-0.5cm} Lefschetz $(1,1)$-theorem in tropical geometry}
\titlemark{Lefschetz $(1,1)$-theorem in tropical geometry}
\author{\vspace{0cm} Philipp Jell, Johannes Rau, and Kristin Shaw}
\authoraddresses{
\authordata{Philipp Jell}{\firstname{Philipp} \lastname{Jell}\\
\institution{Georgia Institute of Technology,
686 Cherry Street,
Atlanta, GA 30332-0160, USA}\\
\email{philipp.jell@math.gatech.edu}}\\
\authordata{Johannes Rau}{\firstname{Johannes} \lastname{Rau}\\
\institution{Universit\"at T\"ubingen,
Geschwister-Scholl-Platz, 72074 T\"ubingen, Germany}\\
\email{johannes.rau@math.uni-tuebingen.de}}\\
\authordata{Kristin Shaw}{\firstname{Kristin} \lastname{Shaw}\\
\institution{Department of Mathematics, University of Oslo, Box 1053, Blindern, 0316 Oslo, Norway}\\
\email{krisshaw@math.uio.no}}
}
\authormark{P. Jell, J. Rau, and. K. Shaw}
\date{\vspace{-5ex}} % Empty date or tweak it according to your needs
\journal{\'Epijournal de G\'eom\'etrie Alg\'ebrique} % Epijournal name
\acceptation{This is a \emph{postpublished} version with minor modifications and corrections.\\
The published version is available at \href{https://epiga.episciences.org/4981}{https://epiga.episciences.org/4981}.}

\newcommand{\articlereference}{Volume 2 (2018), Article Nr. 11}

\acknowledgement{The first author was for part of this work supported by the CRC 1085 ``Higher Invariants'' by the Deutsche Forschungsgemeinschaft 
and for the other part by the DFG Research Fellowship  JE 856/1-1. The third author's research was partially  supported by the Bergen Research Foundation project ``Algebraic and topological cycles in tropical and complex geometry''.}

 \usepackage[all]{xy}

\allowdisplaybreaks

\numberwithin{equation}{numsection}
\usepackage{tikz, etoolbox}
\usepackage{pgf,tikz,pgfplots}
\usepackage{graphicx}
\usepackage{comment}

\usetikzlibrary{matrix,calc,shapes,decorations}

\pgfdeclarelayer{background}
\pgfdeclarelayer{main}
\pgfdeclarelayer{foreground}
\pgfsetlayers{background,main,foreground}

\usetikzlibrary{arrows}

\usetikzlibrary{calc}

\newcommand{\ZZ}{{\mathbb Z}}
\newcommand{\Z}{{\mathbb Z}}

\newcommand{\CC}{{\mathcal C}}
\newcommand{\RR}{{\mathbb R}}
\newcommand{\R}{{\mathbb R}}
\newcommand{\QQ}{{\mathbb Q}}

\newcommand{\TT}{{\mathbb T}}
\newcommand{\T}{{\mathbb T}}

\renewcommand{\D}{{\mathcal D}}
\newcommand{\C}{{\mathcal C}}

\newcommand{\F}{{\mathcal F}}

\newcommand{\FS}{\mathcal{F}}
\newcommand{\GS}{\mathcal{G}}
\newcommand{\DS}{\mathcal{D}}

\newcommand{\US}{\mathcal{U}}
\newcommand{\ES}{\mathcal{E}}
\newcommand{\MS}{\mathcal{M}}

\renewcommand{\div}{\mathrm{div}}

\newcommand{\Linear}{\mathbb{L}}
\newcommand{\DD}{\mathcal{D}}

\newcommand{\diff}{d}
\newcommand{\phihat}{\hat{\phi}}
\newcommand{\face}{\lhd}
\newcommand{\imap}{\iota_{\tau,\sigma}}
\newcommand{\rmap}{\rho_{\tau,\sigma}}
\newcommand{\imapdd}{\iota_{\Delta',\Delta}}
\newcommand{\BF}{{\bf F}}

\newcommand{\clX}{\overline{V}}
\newcommand{\lift}[1]{\widetilde{#1}}
\newcommand{\bdry}[1]{#1_\infty}
\newcommand{\down}[1]{#1_\leq}

\DeclareMathOperator{\Pic}{Pic}

\DeclareMathOperator{\Aff}{Aff}
\DeclareMathOperator{\cyc}{cyc}

\DeclareMathOperator{\relint}{relint}
\DeclareMathOperator{\sed}{sed}

\DeclareMathOperator{\Hom}{Hom}

\DeclareMathOperator{\St}{St}
\DeclareMathOperator{\dist}{dist}
\DeclareMathOperator{\gen}{gen}
\DeclareMathOperator{\Ker}{Ker}
\DeclareMathOperator{\ch}{ch}
\DeclareMathOperator{\ND}{D}

\DeclareMathOperator{\rank}{rank}
\DeclareMathOperator{\id}{id}
\DeclareMathOperator{\BM}{BM}

\DeclareMathOperator{\Div}{CaDiv}

\DeclareMathOperator{\HH}{H}
\DeclareMathOperator{\ZY}{Z}
\DeclareMathOperator{\cell}{cell}
\DeclareMathOperator{\MM}{M}

\newtheorem{thm}{Theorem}[section]

\newtheorem{defn}[thm]{Definition}
\newtheorem{definition}[thm]{Definition}
\newtheorem{prop}[thm]{Proposition}
\newtheorem{proposition}[thm]{Proposition}
\newtheorem{lemma}[thm]{Lemma}
\newtheorem{cor}[thm]{Corollary}
\newtheorem{corollary}[thm]{Corollary}

          {%\theoremstyle{definition}
\newtheorem{rem}[thm]{Remark}}
          {%\theoremstyle{definition}

\newtheorem{example}[thm]{Example}
}
\begin{document}

%%%%%%%%%%%%%%%%%%%%%%%%%%%%%%%
% Add the title to the document
%%%%%%%%%%%%%%%%%%%%%%%%%%%%%%%

\maketitle

%\contribution{}

%%%%%%%%%%%%%%%%%%%%%
% Dedication (if any)
%%%%%%%%%%%%%%%%%%%%%
%\dedication{}

%%%%%%%%%%%%%%%%%%%%%%%%%%%%%%%%%%%%%%%%%%%%%%%%%%%%%%%%%%
% Add abstract, Keywords, MSC classification (recommended)
% Never remove prelims section, make it rather empty
%%%%%%%%%%%%%%%%%%%%%%%%%%%%%%%%%%%%%%%%%%%%%%%%%%%%%%%%%%
\begin{prelims}

\vspace{-0.55cm}

\def\abstractname{Abstract}
\abstract{For a tropical manifold  of dimension $n$ we show that  the tropical homology classes of degree $(n-1, n-1)$ 
which arise as fundamental classes of tropical cycles are precisely those in the kernel of the eigenwave map. 
To prove this we establish a tropical  version of the Lefschetz $(1, 1)$-theorem for rational  polyhedral spaces 
that relates tropical line bundles to the kernel of the wave homomorphism on cohomology. 
Our result  for tropical manifolds then follows by combining this with Poincar\'e duality for integral tropical homology.}



\vspace{0.15cm}

\languagesection{Fran\c{c}ais}{%

\vspace{-0.05cm}
\textbf{Titre. Th\'eor\`eme de Lefschetz $(1,1)$ en g\'eom\'etrie tropicale} \commentskip \textbf{R\'esum\'e.} Pour une vari\'et\'e tropicale de dimension $n$, nous montrons que les classes d'homologie tropicale de degr\'e $(n-1,n-1)$ apparaissant comme des classes fondamentales de cycles tropicaux sont exactement celles dans le noyau de l'application d'onde propre. Pour y parvenir, nous \'etablissons une version tropicale du th\'eor\`eme de Lefschetz pour les $(1,1)$-classes dans les espaces poly\'edraux rationnels qui relie les fibr\'es en droites tropicaux au noyau du morphisme d'onde en cohomologie. Notre r\'esultat pour les vari\'et\'es tropicales s'en d\'eduit alors, en combinant cela avec la dualit\'e de Poincar\'e pour l'homologie tropicale enti\`ere.}

\end{prelims}

%%%%%%%%%%%%%%%%%%%%%
% Content begins here
%%%%%%%%%%%%%%%%%%%%%

\newpage

% Add table of contents (optional)
\setcounter{tocdepth}{1} \tableofcontents

\section{Introduction}

The classical Lefschetz (1,1)-theorem characterises the cohomology classes of complex projective varieties 
which arise  as Chern classes of complex line bundles. 
The theorem asserts that  these classes are precisely the  integral classes in the (1,1)-part of the Hodge decomposition.
It implies the Hodge conjecture (over $\Z$) 
for  the degree $2$ cohomology classes of a complex projective variety. 
In this paper we establish analogous results for rational polyhedral and tropical spaces.

Tropical homology in the sense of Itenberg, Mikhalkin, Katzarkov, and Zharkov  was introduced as an invariant of tropical varieties 
capable of providing  Hodge theoretic information about complex projective varieties via their tropicalisations \cite{IKMZ}.
Tropical homology groups with coefficients in a ring $Q$ can be defined for any rational polyhedral space $X$, see Definition \ref{def:polyhedralspace}. 
The tropical homology  groups with $Q$ coefficients of a rational polyhedral space $X$ are denoted  by   $\HH_{p, q} (X , Q)$. 
We also consider the tropical Borel-Moore homology groups, which are denoted  by $\HH_{p, q}^{\BM}(X, Q)$. 
The corresponding Borel-Moore and usual tropical homology groups agree when $X$ is compact.

A tropical cycle, synonymously a tropical space,  is a rational polyhedral space that satisfies  a balancing condition which is ubiquitous in tropical geometry, see Definition \ref{def:cycle}. Tropical cycles are the candidates for tropicalisations of classical algebraic cycles.
To a tropical cycle  $Z$ of dimension $k$ in a rational polyhedral space $X$, we can associate a tropical homology class which we call the 
fundamental class and denote by $[Z] \in \HH_{k, k}^{\BM}(X, \Z)$.

Tropical manifolds are tropical spaces which are locally modelled on matroidal fans, see Definition \ref{def:tropicalmanifold}.
In this paper, we determine exactly which tropical homology classes in $\HH^{\BM}_{n-1, n-1}(X, \Z)$ 
of a tropical manifold  $X$ of dimension $n$ arise from codimension one tropical cycles. 
 In order to characterise these tropical homology classes, we make use of the  wave homomorphism 
 \begin{align*}
\phihat \colon \HH^{\BM}_{p,q}(X, \Z) \to \HH^{\BM}_{p+1,q-1}(X, \R),
\end{align*}
introduced by Mikhalkin and Zharkov  \cite{MikZhar}, 
which is defined for any rational polyhedral space $X$.  
When $X$ arises as the tropicalisation of 
a family of complex projective varieties and satisfies some additional assumptions, then the wave homomorphism is related to
 the monodromy 
operator on the mixed Hodge structure of the family \cite[Section 7]{MikZhar}. 
 Liu constructed an analogous operator on tropical Dolbeault cohomology of non-archimedean analytic spaces, 
 which he relates to the monodromy operator in the weight spectral sequence \cite{Yifeng}.

It was pointed out by Mikhalkin and Zharkov that the  fundamental class of a tropical cycle in $X$ is in the kernel of $\phihat$.

\begin{thm} \label{Theorem III intro}
For a  tropical manifold $X$  of dimension $n$ the kernel of the wave homomorphism 
\begin{align*}
\phihat \colon \HH^{\BM}_{n-1, n-1}(X,  \mathbb{Z})  \to \HH^{\BM}_{n, n-2}(X, \R)
\end{align*}
consists precisely of the fundamental classes of  tropical cycles of codimension one in $X$. 
\end{thm}

To prove Theorem \ref{Theorem III intro}, we first establish for rational polyhedral spaces, an analogue of the line bundle version of the Lefschetz $(1, 1)$-theorem.
To do so, we consider the sheaf $\Aff_{\Z}$ of integral affine functions. These functions 
 play the role of invertible regular functions in tropical geometry.
We also consider tropical cohomology groups $\HH^{p,q}(X, Q)$, 
which are  the cohomology groups of a  sheaf  
$\FS^p_Q$ on $X$. 
The tropical Picard group of $X$ is defined to be $\Pic(X) := \HH^{1}(X, \Aff_ \Z)$ and 
there is a Chern class map  $c_1 \colon \Pic(X) \to \HH^{1, 1}(X, \Z)$, see Definition \ref{def:Pic}.
These notions in tropical geometry have also appeared in the context of curves \cite{MikZha:Theta} and tropical complexes \cite{Dustin, DustinSurfaces}. 
Definition \ref{def:eigenwave} also describes the wave homomorphism  on tropical cohomology namely, 
$\phi \colon \HH^{p, q}(X, \Z) \to \HH^{p-1, q+1}(X, \R)$.

\begin{thm} \label{Theorem I intro}
Let $X$ be a  rational polyhedral space with polyhedral structure, then  the  image of $c_1 \colon \Pic(X) \rightarrow \HH^{1, 1}(X, \mathbb{Z})$ is equal to 
the kernel of the wave homomorphism 
$\phi \colon \HH^{1, 1}(X, \mathbb{Z})  \to \HH^{0,2}(X, \R)$. 
\end{thm}

To prove Theorem \ref{Theorem I intro},  we use a short exact sequence of sheaves $0 \to \R \to \Aff_{\Z} \to \F^1_{\Z} \to 0$, known as the tropical exponential sequence \cite{MikZha:Theta}. This produces a long exact sequence in cohomology: 
$$ \dots \to \Pic(X) \rightarrow \HH^{1, 1}(X, \mathbb{Z}) \to \HH^{0,2}(X, \R) \to \dots . $$
For $p = 0$, the sheaf $\mathcal{F}^0_{\R}$ is the constant sheaf $\underline{\R}$, 
so we can identify $\HH^{2}(X, \R)$ and $\HH^{0, 2}(X, \mathbb{R})$. 
In Proposition \ref{boundary and wave}, we show that the boundary map $\delta \colon  \HH^{1, q}(X, \mathbb{Z}) \to \HH^{q+1}(X, \R)$
coincides up to sign with the wave homomorphism. For $q=1$, this implies Theorem \ref{Theorem I intro}.

When $X$ is an abstract tropical space of dimension $n$, the cap product with its fundamental class provides  a map
\begin{align}\label{capmap}
\cap [X] \colon \HH^{p,q}(X, \Z) \to \HH^{\BM}_{n-p,n-q}(X, \Z).
\end{align}
This allows us to describe the kernel of the wave homomorphism on homology groups. 

\begin{thm} \label{Theorem II intro}
Let $X$ be a tropical space of dimension $n$. 
$\alpha \in \HH^{1, 1}(X; \mathbb{Z})$ is  such that $\phi(\alpha) = 0$, 
then $\alpha \cap [X] \in \HH^{\BM}_{n-1, n-1}(X, \Z)$ is the fundamental class of a codimension one tropical cycle in $X$. 
\end{thm}

To prove Theorem \ref{Theorem II intro} we first show that any element  
$L \in \Pic(X)$ has a rational  section in the sense of Definition \ref{def:section}. 
A tropical Cartier divisor is a tropical line bundle  $L \in \Pic(X)$ together with a section  $s$. 
We can then define a map  $\div \colon \Div(X) \to \ZY_{n-1}(X)$, where $\Div(X)$ is the group of Cartier divisors on $X$ and 
$\ZY_{n-1}(X)$ is the group of dimension one tropical cycles in $X$.  
We then show that the map given by capping with the fundamental class  (\ref{capmap})  is an isomorphism when $X$ is a tropical manifold. 
This extends the version of Poincar\'e duality with real coefficients  of  Smacka and the first and third authors \cite[Theorem 2]{JSS}. 
Combining this statement with Theorem \ref{Theorem II intro}, we are able to prove Theorem \ref{Theorem III intro}.

The last section presents corollaries and examples of our main theorems. In particular, we consider tropical abelian surfaces and Klein bottles with a tropical structure.   We also calculate the wave map for two combinatorial types of smooth tropical quartic  surfaces. The Picard rank  of a polyhedral space $X$ is defined  to be  the rank of $\Pic(X)$. 
We prove the following statement  regarding the Picard ranks of smooth tropical quartic surfaces. 

\begin{thm}\label{Prop Intro}
For every $1 \leq \rho \leq 19$ there exists a smooth tropical quartic surface with Picard rank $\rho$. Moreover, such surfaces can be chosen to have the same combinatorial type. 
\end{thm}

\subsection*{Acknowledgements}

We are very grateful to  Ilia Itenberg, Grigory Mikhalkin and Ilia Zharkov for useful discussions and also to 
Edvard Aksnes, Andreas Gross, Arthur Renaudineau and Yuto Yamamoto for helpful comments and 
corrections on earlier versions of this paper.
We also wish to thank an anonymous referee for their comments which helped us to improve this paper.
We would also like to thank the Max Planck Institute Leipzig for hosting us during a part of this collaboration. 
This project was started while the last two authors were visiting the University of Geneva. 
We would like to thank Grigory Mikhalkin for the kind invitation.

\section{Preliminaries}

We set $\TT = [-\infty, \infty)$ and equip this set with the topology whose basis consists of the intervals 
$[-\infty , b)$ and $(a, b)$ for $a,b \neq -\infty$. We equip $\TT^r$ with the product topology. 
The set $\TT^r$ is a stratified space. For a subset $I \subset [r]$ define 
$\R^r_I  = \{x \in \TT^r \ | \ x_i = -\infty \Leftrightarrow i \in I\}$ and 
$\T^r_I$ is the closure of $\R^r_I$ in $\TT^r$. 
We then have $\R^r_I \cong \R^{r-|I|}$ and $\TT_I^r \cong \TT^{r-|I|}$. The \emph{sedentarity} of a point  
$x \in \TT^r$ is  $\sed(x) := \{ i \in [n] \ | \ x_i = -\infty\}$. 

\subsection{Abstract polyhedral spaces and tropical varieties}

A \emph{rational polyhedron} in $\R^r$ is a subset defined by a finite system of affine (non-strict)  inequalities 
$\langle w_i, v \rangle \geq c_i$ with $c_i \in \R$ and $w_i \in \Z^r$.
A \emph{face} of a polyhedron $\sigma$ is a polyhedron which is obtained by turning some of the defining inequalities of $\sigma$ into equalities.

A \emph{rational polyhedron} in $\TT^r$ is the closure of a rational polyhedron in 
$\R^r_I \cong \R^{r - |I|} \subset \T^r$ for some $I\subset[r]$. 
A \emph{face} of a polyhedron $\sigma$ in $\T^r$ is the closure of a face of $\sigma \cap \R_J$ for some $J \subset [r]$.
A \emph{rational polyhedral complex} $\CC$ in $\TT^r$ is a finite set of polyhedra in $\TT^r$,
satisfying the following properties:
\begin{itemize}
\item[\rm (1)]
For  a polyhedron $\sigma \in \CC$, if $\tau$ is a face of $\sigma$ (denoted $\tau \prec \sigma$) we have $\tau \in \CC$. 
\item[\rm (2)]
For $\sigma, \sigma' \in \CC$, if $\tau = \sigma \cap \sigma'$ is non-empty, then $\tau$ is a face of both $\sigma $ and $\sigma'$. 
\end{itemize}
The maximal polyhedra, with respect to inclusion, are called facets.
If all facets of $\CC$ have the same dimension $n$, we say $\CC$ is of pure dimension $n$. 
The support of a polyhedral complex $\CC$  is the union of all its polyhedra and is denoted $\vert \CC \vert$. 
If $X = |\CC|$, then $X$ is called a \emph{rational polyhedral subspace} of $\T^r$ and 
$\CC$ is called a \emph{rational polyhedral structure} on $X$.

The relative interior of a polyhedron $\sigma$ in $\T^r$, denoted $\relint(\sigma)$, is defined to be the set obtained after removing
all of the proper faces of $\sigma$.
Given a polyhedral complex $\CC$ in $\T^r$, for  $\sigma \in \CC$, the closed star of $\sigma$ is
$\overline{\St}(\sigma) := \{\tau \in \CC \ | \ \exists \sigma' \in \CC \text{ such that } \tau,\sigma \subset \sigma'\}$. 
The open star $\St(\sigma)$ of $\sigma$ is the open set which is the relative interior of the support of $\overline{\St}(\sigma)$. 
Also, let $\CC_I$ denote the union of polyhedra 
$\sigma \in \CC$ for which $\relint(\sigma) \subset \R^r_I$. 
For  a rational polyhedron  $\sigma$ in $\T^r$, 
we denote  $\sigma \cap \R^r_I$  by $\sigma_I$.  

A map $f \colon M \to N$, where  $M \subset \T^m$ and $N \subset \T^n$, is an
\emph{extended affine $\Z$-linear map} if it is continuous 
and there exist $A \in \text{Mat}(n \times m, \Z)$,
$b \in \R^n$ such that $f(x) = Ax +b$ for all $x \in \R^m$.

\begin{defn}
\label{def:polyhedralspace}
A \emph{rational polyhedral space} $X$ is a paracompact, second countable Hausdorff topological space with an atlas of charts 
$(\varphi_{\alpha} \colon U_{\alpha} \rightarrow \Omega_{\alpha} \subset X_{\alpha})_{{\alpha} \in A}$ such that:
\begin{itemize}
\item[\rm (1)]
The $U_{\alpha}$ are open subsets of $X$, the $\Omega_{\alpha}$ are open subsets of rational polyhedral subspaces 
$X_{\alpha} \subset \T^{r_{\alpha}}$, and the maps 
$\varphi_{\alpha} \colon U_{\alpha} \rightarrow \Omega_{\alpha}$ are homeomorphisms for all $\alpha$; 
\item[\rm (2)]
for all $\alpha, \beta \in A$ the transition map
\begin{align*}
\varphi_{\alpha}\circ \varphi^{-1}_ \beta\colon \varphi_ \beta(U_{\alpha} \cap U_ \beta) \rightarrow {\varphi_\alpha(U_{\alpha} \cap U_ \beta)}
\end{align*}
are  extended affine $\Z$-linear maps.
\end{itemize} 
\end{defn} 

\begin{defn}\label{def:facestructure}
Let $X$ be a rational polyhedral space.
A \emph{rational polyhedral structure}
on $X$ is a finite family of closed subsets $\CC$ such that the following conditions hold:
\begin{itemize}
\item[\rm (1)]
$X = \bigcup_{\sigma \in \CC} \sigma$;
\item[\rm (2)]
for each $\sigma$ there exists a chart $\varphi_\sigma \colon U \rightarrow \Omega \subset X$
such that $\overline{\St}(\sigma) \subset U$ and $\{\varphi_\sigma(\tau) \ | \ \tau \in \overline{\St}(\sigma)\}$ is a rational polyhedral complex
in $\T^s \times \R^{r-s}$.
\end{itemize}
\end{defn}

\subsection{Multi-(co)tangent (co)sheaves} 

Let  $\C$ be  a  rational polyhedral complex in $\TT^r$. 
For a face  $\sigma \in \C_I$,  denote by $\mathbb{L} (\sigma) \subset \R^r_I$
the subspace generated by the vectors in $\R^r_I$ tangent to $\sigma$.
Set $\mathbb{L}_{\Z} (\sigma) = \mathbb{L} (\sigma) \cap \Z^r_I \subset \Z^r_I$.

\begin{defn}\label{def:multitangent}
 For $\sigma \in \CC_I$,
the  $p$-th \emph{integral multi-tangent} and \emph{integral multi-cotangent space}
 of $\CC$ at $\sigma$ are the $\Z$-modules
$${\bf F}^{\Z}_p(\sigma) =  \left(\sum \limits_{ \sigma' \in \CC_I: \sigma \prec \sigma'} \bigwedge^p \Linear(\sigma')\right) \cap \bigwedge^p \Z^r_I 
\qquad \text{and} \qquad
{\bf F}_{\Z}^p(\sigma) = \Hom({\bf F}^{\Z}_p(\sigma),\Z)$$
respectively.  
If  $\tau$ is a face of $\sigma$ there are natural  maps 
\begin{align*}
\imap \colon {\bf F}^{\Z}_p(\sigma) \to {\bf F}^{\Z}_p(\tau) \quad \text{ and } 
\quad  \rmap \colon {\bf F}_{\Z}^p(\tau) \to {\bf F}_{\Z}^p(\sigma).
\end{align*}
\end{defn}

For $Q$ any ring such that $\Z \subset Q \subset \R$ define 
${\bf F}_Q^{p}(\sigma) ={ \bf F}_{\Z}^{p}(\sigma) \otimes Q$ and ${\bf F}^Q_{p}(\sigma) ={ \bf F}^{\Z}_{p}(\sigma) \otimes Q$.
When $Q = \R$ we drop the use of the sup- and  sub-scripts on ${\bf F}_p(\sigma)$ and ${\bf F}^p(\sigma)$, respectively.

From the  
$\Z$-modules ${\bf F}_{\Z}^p(\sigma)$,  it is possible to construct a sheaf on $|\CC| \subset \T^r$
following   \cite[Section 2.3]{MikZhar}. 
 For each open set $\Omega \subset |\CC|$, consider the  poset $P(\Omega)$ 
whose elements are the connected components $\sigma$ of faces of $\CC$ intersecting with $\Omega$. 
The elements of $P(\Omega)$  are ordered by inclusion and if $\tau \prec \sigma$ recall there are maps 
$\rmap  \colon {\bf{F}}_{Q}^p(\tau) \to {\bf{F}}_{Q}^p(\sigma)$.

 \begin{defn}[\cite{MikZhar}]\label{def:sheafconst}
Let $\C$ be a rational polyhedral complex of $\T^r$. For an open set $\Omega \subset |\CC|$ define the vector space
\begin{align*}
\mathcal{F}^p_{Q}(\Omega)  := \varprojlim_{\sigma \in P(\Omega)} {\bf F}_{Q}^p( \sigma).
\end{align*}
\end{defn}

The sheaves $\F^p_Q$ are constructible and  do not depend on the polyhedral structure $\CC$ but only on the support $|\CC|$. 
For a polyhedral space $X$, the sheaves $\mathcal{F}^p$ are defined by gluing along charts. 
In fact, this definition does not require a polyhedral structure on $X$,  see \cite{JSS}.

\subsection{Tropical (co)homology}

In the following we always assume that $X$ is a rational polyhedral space 
which admits a rational polyhedral structure $\CC$.
In this case, the $\Z$-modules ${\bf F}^{\Z}_p(\sigma)$ and ${\bf F}_{\Z}^p(\sigma)$
and the maps $\imap$, $\rmap$ are well-defined for any $\tau \prec \sigma \in \CC$.

We let $\Delta_q$ denote an abstract $q$-dimensional simplex. 
Again $Q$ will be a ring satisfying  $\Z \subset Q \subset \R$.

\begin{defn}\label{def:stratsimp}
A \emph{$\CC$-stratified $q$-simplex} in $X$ is a continuous map $\delta \colon \Delta_q\to X$
such that
\begin{itemize}
\item 
for each face $\Delta' \subset \Delta_q$, we have  $\delta(\relint(\Delta')) \subset \relint(\tau)$ for some $\tau \in \CC$; 
\item 
if $\Delta_q = [0, \ldots, q]$ and $\varphi$ is a chart containing 
$\delta(\Delta_q)$, then 
\begin{align*}
\sed(\varphi(\delta(0))) \supset \sed(\varphi(\delta(1))) \supset \ldots \supset \sed(\varphi(\delta(q))).
\end{align*}
\end{itemize} 

For $\tau \in \CC$ let $C_q(\tau)$ denote the abelian group generated by stratified $q$-simplices 
$\delta \colon \Delta_q \to X$ that satisfy $\relint(\Delta_q) \subset \relint(\tau)$. 
\end{defn}

\begin{defn}\label{def:tropchainhomo}
The groups of \emph{tropical} $(p,q)$-\emph{chains} and \emph{cochains} 
 with respect to $\CC$ and with $Q$-coefficients are respectively, 
\begin{align}
  C_{p,q}(X, Q) &:= \bigoplus_{\tau \in \CC} {\bf F}_p^{Q}(\tau) \otimes_{\Z} C_q(\tau),  \\
	C^{p,q}(X, Q) &:= \Hom_Q(C_{p,q}(X, Q), Q) = \bigoplus_{\tau \in \CC} {\bf F}^p_{Q}(\tau) \otimes_{\Z} \Hom_\Z(C_q(\tau),\Z). 
\end{align}

The complexes of tropical $(p, \bullet)$-chains and cochains are respectively, 
\[(C_{p,\bullet}(X, Q), \partial) \qquad \text{ and } \qquad  (C^{p,\bullet}(X, Q), d)\]
where the $\partial$ and $d$ are the usual singular differentials composed, when necessary, with $\imap$ and $\rmap$ respectively. 

The  \emph{tropical homology} and \emph{tropical cohomology} groups with coefficients in $Q$ are respectively, 
\begin{align*}
\HH_{p,q}(X, Q) := \HH_q( C_{p,\bullet}(X, Q)) \quad \text{ and } \quad 
\HH^{p,q}(X, Q) := \HH^q( C^{p,\bullet}(X, Q)). 
\end{align*}
\end{defn}

\begin{defn}\label{Borel-Moore}
The \emph{tropical Borel-Moore} chain groups $C_{p,q}^{\BM}(X, Q)$ consist of formal infinite sums
of elements of $C_{p, q}(X, Q)$ with the condition that locally only finitely many simplices have non-zero coefficients. 

The \emph{tropical Borel-Moore homology} groups  are denoted   $\HH_{p,q}^{\BM}(X, Q)$. 
They  are the homology groups of the complex 
$(C_{p,\bullet}^{\BM}(X, Q), \partial)$.
\end{defn}

\begin{rem} \label{simplicialhomology}{\rm
  For computations, we will often use the simplicial version of the tropical  (co)homology groups defined above. 
	It is possible to construct a locally finite
	simplicial structure $\DD$ on $X$ such that all simplices are $\CC$-stratified, see  \cite[Section 2.2]{MikZhar}. 
	We call such a structure a $\CC$-stratified simplicial structure. 
	Following the standard conventions for simplicial (co)homology, we obtain
	simplicial tropical homology and cohomology  groups, as well as the Borel-Moore 
	variants respectively. 

	The equivalence of singular and simplicial homology with ${\bf F}_p^{Q}$-coefficients
	is proved in \cite[Section 2.2]{MikZhar}. 
	The argument uses cellular homology as an intermediate step, which is introduced
	in Section \ref{sec:poincare} in this work. 
	The argument, which uses barycentric subdivisions, still applies to our case thanks to 
	the fact that $\DD$ is a $\CC$-stratified simplicial structure. 
	The discussion in \cite{MikZhar} is restricted to  standard homology, but the arguments can 
	be extended to the Borel-Moore case after noting that the cellular homotopy argument still applies
	and that the cellular chain complex can still be described in terms of relative singular homology
	\[
		C_{p,q}^{\BM, \cell}(\CC, Q) = H_{p,q}^{\BM}(X^q, X^{q-1}; Q).
	\]
	Here, $X^q$ denotes the support of the $q$-skeleton of $\CC$ and $H_{p,q}(X^q, X^{q-1}; Q)$
	denotes relative homology.
	We use the identification of singular and simplicial homology throughout the rest of the text.
	We also use the same notation to denote both variants of the tropical (co)homology groups. }
\end{rem}

\subsection{The eigenwave homomorphism}

Throughout this section $X$ is a rational  polyhedral space equipped with a rational  polyhedral structure $\CC$. 
Before presenting the definition of the eigenwave homomorphism from \cite{MikZhar} we provide some notation. 
If  $\delta \colon [0, \dots, q+1] \to X$ is a $\CC$-stratified $q+1$-simplex, 
we denote the restriction of $\delta$ to the face $[0, \dots, q]$ by $\delta_{0 \dots q}$
and by $\sigma$ and $\tau$ the faces of $\CC$ containing the image of the relative interior of $[0,\dots,q+1]$ and $[0,\dots,q]$, respectively.
Moreover, the  vector   
$v_{\delta[q,q+1]} \in {\bf F}_1(\tau)$ is defined to be  the difference of the endpoints  of $\delta_{q,q+1}$ in a chart $\varphi$
containing  $\sigma$.   
More precisely, 
\begin{align}\label{eqn:wavevector}
v_{\delta[q,q+1]} := \imap(\varphi(\delta(q+1))) - \varphi(\delta(q)).
\end{align}
The vector  $v_{\delta[q,q+1]}$ is in the linear space $\mathbb{L}_{\Z}(\tau) \otimes \R$. Moreover, 
given a vector $w \in \mathbf{F}^{\Z}_{p-1}(\tau)$ we have  $w \wedge v_{\delta[q,q+1]} \in \mathbf{F}_p(\tau)$. 
\begin{defn} \label{def:eigenwave}
The eigenwave  homomorphism on singular tropical chains, 
\begin{align*}
\phihat \colon C_{p-1,q+1}(X, \Z) \to C_{p, q}(X, \R),
\end{align*}
is defined on a tropical $(p,q)$-cell $v \otimes \delta$ to be
\begin{align*}
\phihat(v \otimes \delta)  = (\imap(v) \wedge v_{\delta[q,q+1]}) \otimes \delta_{0 \dots q}.
\end{align*} 
Dually, the eigenwave  homomorphism on singular tropical cochains, 
\begin{align*}
\phi \colon C^{p,q}(X, \Z) \to C^{p-1, q+1}(X, \R),
\end{align*}
is defined on a tropical $(p,q)$-cocell $\alpha$ to be
\begin{align*}
\phi(\alpha)(v \otimes \delta)  = \alpha(\phihat(v \otimes \delta)) = 
\alpha \left( (\imap(v) \wedge v_{\delta[q,q+1]}) \otimes \delta_{0 \dots q} \right).
\end{align*} 
\end{defn}

A \ direct \ computation \ shows \ that \ these \ give \ morphisms \
$\phihat \colon C_{p-1,\bullet}(X, \Z)[1] \to C_{p,\bullet}(X, \Z)$ \ and 
$\phi \colon C^{p,\bullet}(X, \Z) \to C^{p-1, \bullet}(X, \R)[1]$.
 Therefore $\phihat$ and $\phi$ descend to maps  on homology and  cohomology, 
which we also denote by $\phihat$ and $\phi$, 
respectively.

\section{Tropical exponential sequence}\label{sec:seq}
Here we will prove Theorem \ref{Theorem I intro} using the tropical exponential sequence $(\ref{exponential sequence})$. 
Throughout this section $X$ is a rational polyhedral space with a rational polyhedral structure $\CC$.

\begin{defn} 
The sheaf of real valued functions on $X$ which are affine with integral slope in each chart is denoted by $\Aff_\Z$.  
\end{defn}

\begin{defn}
Let $x$ be a point in a polyhedral space $X$ 
and  $\varphi \colon U \to \T^r$ a chart such that $x \in U$ and $\sed(\varphi(x)) \neq \emptyset$. 
Then 
$v \in \R^r$ is a \emph{divisorial direction at $x$} 
if there exists an $x_0 \in U$ 
with $\sed \varphi (x_0) = \emptyset$ such that for all $t<0$ we have 
$x_t  = \varphi(x_0) +tv  \in \varphi(U)$ and $\lim_{t \to \infty}  \varphi(x_0) + tv   = x$.  
\end{defn}

Note that any affine function $f \in \Aff(U)$ is constant  along the divisorial directions to any $x \in U$ since the value  $f(x)$ is a real number.
Taking the differential of a real valued function  provides a surjective map 
$d \colon \Aff_\Z \to \FS_\Z^1$.  
The kernel is the sheaf of locally 
constant real functions $\underline{\R}$. 
The  \emph{tropical exponential sequence} is 
\begin{align} \label{exponential sequence}
0 \to \underline{\RR} \to \Aff_{\Z} \to \F_{\Z}^1 \to 0.
\end{align}
After passing to the long exact sequence in cohomology for all $q$ there is the  coboundary map,
\begin{align*}
\delta \colon \HH^{1,q}(X; {\mathbb{Z}}) \to \HH^{q+1}(X,\RR).
\end{align*}
Recall that $\mathcal{F}^0$ is the constant sheaf  $\underline{\R}$ for all $X$, 
therefore we identify  $\HH^q(X,\RR)$ and $\HH^{0,q}(X, \RR)$.

\begin{lemma} \label{lem acyclic}
Let $\DD$ be a $\CC$-stratified simplicial structure and
 $\mathcal{U}$ denote the cover of $X$ by the open stars of vertices of either $\CC$ or $\DD$. 
Then $\mathcal{U}$ is a Leray cover of $X$ for the sheaves $\underline{\R}$, $\Aff_\Z$ and $\FS^{p}_\Z$. 
\end{lemma}
\begin{proof}
Firstly, we show acyclicity of any open star $U$ of a face for the sheaves $\underline{\R}$, $\Aff_\Z$ and $\FS^{p}_\Z$. 
The open set $U$ is contractible, thus acyclic for $\underline{\R}$. 
Furthermore, the contraction can be chosen so that it respects the simplicial structure on $U$. 
Following the arguments in the proof of \cite[Proposition 3.11]{JSS}, we see that $U$ is acyclic for $\FS^1_\Z$. 
The long exact sequence associated to (\ref{exponential sequence}) implies that  $U$ is acyclic for $\Aff_\Z$ as well. 

The intersection of stars of vertices is the star of the minimal face containing these vertices. Therefore, all intersections of the cover are acyclic and 
 $\mathcal{U}$ is a Leray cover of $X$.
 \qed
\end{proof}

\begin{rem} \label{rem:Cech vs simplicial}{\rm
Let $\mathcal{U}$ be the open cover given by stars of vertices of a $\CC$-stratified simplicial structure $\DD$ on $X$. 
Then there is a canonical isomorphism between 
the tropical  simplicial cohomology groups with respect to $\DD$ 
and the \v{C}ech cohomology of the sheaves $\mathcal{F}^p_Q$ with respect to the cover $\mathcal{U}$.
The \v{C}ech chain group $C^q(\FS^p_Q, \mathcal{U})$ 
is canonically isomorphic to 
the group of $q$-simplicial cochains with coefficients in ${\bf F}^p_Q$, 
since $\FS^p_Q(U_{i_0,\dots,i_q}) = {\bf F}^p_Q([i_0,\dots,i_q])$ for any $q$-simplex $[i_0,\dots,i_q] \in \DD$.
Also the differential maps in both cases agree.  We also use this identification of simplicial and \v{C}ech cohomology groups throughout the following sections  without using  different notations.}
\end{rem}

\begin{proposition} \label{boundary and wave}
The coboundary map 
$\delta \colon \HH^{1,q}(X, \Z) \to \HH^{0, q+1}(X,\RR)$ 
coincides, up to sign, with the  eigenwave  homomorphism. More precisely, we have 
$
  \delta = (-1)^{q+1} \phi.
$
\end{proposition}

\begin{proof}
Let $\DD$ be a stratified simplicial structure on $X$. 
Let $\D_q$ denote the simplicies of $\DD$ of dimension $q$. 
Write $[i_0,\dots,i_q]$ for the $q$-simplex with vertices $i_0,\dots,i_q \in \DD$ with the orientation induced by the ordering of the vertices.  
For a $q$-simplex $[i_0,\dots,i_q]$, denote its open star by $U_{i_0\dots i_q}$. 
 
We will compare the coboundary and the eigenwave maps using  \v{C}ech cochains  with respect to the cover $(U_{i})_{i \in \D_0}$. 
An element $\alpha \in \HH^{1,q}(X, \Z)$ is given by a tuple $(\alpha_{i_0\dotsi q})_{[i_0,\dots,i_q] \in \D_{q}}$ 
where $\alpha_{i_0\dots i_q}  \in \FS^1_{\Z}(U_{i_0 \dots i_q})$. 
We choose a collection of functions  $f_{i_0 \dots i_q} \in \Aff_{\Z}(U_{i_0 \dots i_k})$ such that $d f_{i_0 \dots i_q} = 
\alpha_{i_0 \dots i_q} $ for all $[i_0,\dots,i_q] \in \DD_q$. 
Since  the functions $f_{i_0\dots i_q}$ are integer affine and the vertex $i_q$ has minimal sedentarity among all of $i_0 \dots i_q$, 
each function  $f_{i_0 \dots i_q}$ extends uniquely by continuity to the vertex $i_q$.  
We    normalise our choices  in such a way that $f_{i_0 \dots i_q}(i_q) = 0$. 

Write $f = (f_{i_0 \dots i_q})_{[i_0,\dots,i_q] \in \D_{q}}$.
Since $(\alpha_{i_0 \dots i_q})$ is a closed \v{C}ech chain, 
 the \v{C}ech boundary
\begin{align*}
(\partial f)_{i_0 \dots i_{q+1}} = \sum (-1)^k f_{i_0 \dots \widehat{i_{k}} \dots i_{q+1}}
\end{align*}
is a constant function. 
To compute this constant, we evaluate at $i_{q+1}$ and find
\begin{align*}
(\partial f)_{i_0 \dots i_{q+1}}(i_{q+1}) = (-1)^{q+1} f_{i_0 \dots i_q}(i_{q+1})
\end{align*}
because of our normalisation. 

Note that if $i_{q+1}$ has strictly lower sedentarity than $i_q$, then $f_{i_0 \dots i_q}$ is constant 
when moving along the divisorial direction at $i_{q+1}$  towards $i_q$. 
Let  $\pi_{\tau, \sigma}$ be the projection  (along the divisorial direction) between strata containing the relative interior of two faces  $\tau$ and $\sigma$. 
In particular, if the relative interiors of  $\tau$ and $\sigma$ are contained in the same strata this map is the identity. 
Then, whether or not  $i_{q} $ and $i_{q+1}$ have the same sedentarity, we have
$f_{i_0 \dots i_q}(i_{q+1}) = f_{i_0 \dots i_q}(\pi_{\tau, \sigma} (i_{q+1}))$, 
where $\tau$ and $\sigma$ are the faces containing $[i_0,\dots,i_q]$ and $[i_0,\dots,i_{q+1}]$, respectively. 
Therefore, 
\begin{align*}
\phi(\alpha)_{i_0 \dots i_{q+1}} = \alpha_{i_0 \dots i_q}(v_{\delta[q,q+1]} )  
=  \alpha_{i_0 \dots i_q}(\pi_{\tau, \sigma}(  i_{q+1}) - i_{q} ) 
=  f_{i_0 \dots i_q}(\pi_{\tau, \sigma} ( i_{q+1} ) ) = f_{i_0 \dots i_q}( i_{q+1} ) 
\end{align*}
since $f_{i_0 \dots i_q}(i_{q}) = 0$. 
This completes the proof. \qed
\end{proof}

\begin{defn} \label{def:Pic}
The \emph{tropical Picard group} is  $\Pic(X):= \HH^{1}(X, \Aff_{\Z})$. 
The map from $\Pic(X)$ to  $\HH^{1,1}(X, \Z)$ provided by the tropical exponential sequence is
called the \emph{Chern class map} and is denoted by $c_1 \colon \Pic(X) \to \HH^{1,1}(X, \Z)$.
\end{defn}

\begin{proof}[Proof of Theorem \ref{Theorem I intro}]
The kernel of the boundary map $\delta$ is $\Pic(X) = \HH^1(X, \Aff_{\Z})$ by the long exact sequence associated to (\ref{exponential sequence}) 
and by Propostion \ref{boundary and wave} this is also the kernel of the eigenwave  homomorphism. 
This completes the proof. 
\qed
\end{proof}

\begin{rem}{\rm
There is also a version of Sequence (\ref{exponential sequence}) with real coefficients namely,  
\begin{align} \label{exponential sequenceR}
  0 \to \underline{\RR} \to \Aff \to \F^1 \to 0,
\end{align}
where $\Aff$ denote the sheaf of functions which are affine in each chart, not necessarily with integral slopes.
By the same argument the boundary map of the long exact sequence is equal to the eigenwave map extended to cohomology with $\R$-coefficients 
$\phi \colon \HH^{1, q}(X, \R) \to \HH^{0,q+1}(X, \R)$. 
Mikhalkin and Zharkov  conjecture that 
$\phi^{p-q} \colon \HH^{p,q}(X, \R) \to \HH^{q,p}(X, \R)$ is an isomorphism for all $p \geq q$  \cite[Conjecture 5.3]{MikZhar}. 
This would imply that $\phi \colon \HH^{1,q}(X, \R) \to \HH^{0,q+1}(X, \R)$ is surjective for all $q$.
By the long exact sequence derived from the short exact sequence in (\ref{exponential sequenceR}) 
this happens if and only if $\HH^{0,q}(X, \R) \to \HH^{q}(X, \Aff)$ is zero for all $q \geq 1$, 
which would in turn imply that for all $q$ the following sequence is exact
\begin{align*}
0 \to \HH^{q}(X, \Aff) \to \HH^{1,q}(X, \R) \to \HH^{0,q+1}(X, \R) \to 0.
\end{align*}
A conjecture similar to the one of Mikhalkin and Zharkov was made by Liu \cite{Yifeng} for tropical Dolbeault cohomology,
which is a cohomology theory of non-archimedean analytic spaces defined using superforms in the sense of Lagerberg \cite{Lagerberg}. 
We refer the reader to \cite{CLD, Gubler} for the construction of these forms on analytic spaces and 
to \cite{JSS} for the relation between the cohomology of superforms and the tropical cohomology groups considered here.}
\end{rem}

\section{Tropical cycle class map}

In this section we prove Theorem \ref{Theorem II intro}.  
To do this,  we  first prove the existence of sections of tropical line bundles, 
and that the  construction of the divisor of a section is compatible with the Chern class map combined with capping
with the fundamental class.

\subsection{Tropical line bundles and sections}

Throughout this section $X$ is a rational polyhedral space 
with polyhedral structure $\CC$.

\begin{definition}
	Let $U \subset X$ be an open subset. 
	A \emph{tropical rational  function} $f$ on $U$
	is  a continuous function $f \colon U \to \R$ such that
	for every point $x \in U$ there exists a neighbourhood
	$x \in V \subset U$ and a polyhedral structure $\CC'$
	on $V$ such that 
	$f|_{\sigma}$ is (the restriction of) an affine $\Z$-linear function
	for any $\sigma \in \CC'$. 
	The set of tropical rational functions on $U$
	is denoted by $\MS(U)$.
\end{definition}

The map $U \mapsto \MS(U)$ defines a sheaf on $X$. 
We consider the short exact sequence of sheaves
\begin{align*}
0 \to \Aff \to \MS \to \MS / \Aff \to 0.
\end{align*}
Upon taking the long exact sequence in cohomology we obtain a map 
$\delta \colon \HH^0(X, \MS/\Aff) \to \HH^1(X,\Aff)$.
Recall that $\Pic(X) =\HH^1(X,\Aff)$

\begin{definition}\label{def:section}
  Let $L \in \Pic(X)$ be a line bundle.
	A \emph{section} of $L$ is an element $s \in \HH^0(X, \MS/\Aff)$ such that $\delta(s) = L$.
\end{definition}

  Let us assume that $L$ can be represented 
	by transition functions $(f_{ij})$ with respect to the open covering 
	$\US = (U_i)$. 
	Then a section $s$ of $L$ is equivalent to a collection of tropical rational functions
	$s_i \in \MS(U_i)$ which satisfies
	\[
		s_i - s_j = f_{ij}
	\]
	for all $i \neq j$. 
We use the notation $\Div(X) 	= \HH^0(X, \MS/\Aff)$ and call an element $s \in \Div(X)$
a Cartier divisor of $X$.

In the remainder of this section we establish the existence of a section of a tropical line bundle on a polyhedral space.
A version of this  statement first appeared in the thesis of Torchiani in the case 
when $X$ has no points of sedentarity  \cite[Theorem 2.3.4]{CaroThesis}. 
We start with the following lemmas.

\begin{lemma}\label{lem:piecewiseaffineonsigma}
Let  $\sigma$  be a compact rational polyhedron in $\TT^r$.
If $s \colon  \partial \sigma \to \R$ is a 
rational function, 
then $s$ can be extended to a rational function on all of $\sigma$.
\end{lemma}

\begin{proof}
We start with the case where $\sigma \subset \R^r$, so that  
$\sigma$ does not contain points of higher sedentarity.
We can assume without loss of generality that $\sigma$ is of dimension $r$. 
For a codimension one  face $\tau $ of $\sigma$, let $H_{\tau}$ denote the hyperplane in $\R^r$ containing $\tau$. 
We can construct a rational  function $h_{\tau} \colon H_{\tau} \rightarrow \R$ 
which restricts to $s$ on the face $\tau$. 
To do this, notice that each point in $H_{\tau}$ can be uniquely written as $x + v$ 
where $x \in \delta \prec \tau$ and $v$ 
lies in the normal cone of the face $\delta$ in the  polyhedron $\tau$ with respect to the standard scalar product in $\R^r$. 
Then $h_{\tau}(x + v) = s(x)$. 

For each codimension one  face $\tau$ of  $\sigma$, choose 
a vector $v_{\sigma,\tau} \in \ZZ^r$ pointing from 
$\tau$ to $\sigma$ such that $\Linear_{\Z}(\sigma) = \Linear_{\Z}(\tau) + \ZZ v_{\sigma,\tau}$.
Then let $\pi_{\tau} \colon \R^r \to H_ \tau$   be 
defined by $ \pi_\tau(x)  = x -  \dist(x, H_\tau) v_{\sigma, \tau}$. Choose $m \in \ZZ$ and set
\begin{align*}
  f_{m, \tau}(x) = h_\tau(\pi_\tau(x)) + m \dist(x,H_\tau).
\end{align*}
We will show that for each $\tau$ there exists $m_{\tau} \in \ZZ$ such that $f_{m_\tau,\tau}(x) \leq s(x)$ 
for all $x \in \partial \sigma$. Since $f_{m_\tau,\tau}(x)  = s(x)$ 
for all $x \in \tau$, 
this implies that 
the rational function 
\begin{align}\label{eqn:defh}
  h \colon \sigma \to \R, \ x \mapsto \max_\tau {f_{m_{\tau}, \tau}(x)},
\end{align}
satisfies $h|_{\partial \sigma} = s$, as required.

To find {$m_\tau$} for a fixed $\tau$, 
we proceed as follows.
Let $D \subset H_\tau$ be a domain of linearity
of $h_\tau$ and $\delta \subset \partial \sigma$
be a domain of linearity of $s$. 
We will show that there exists an $m$ such that $f_{m,\tau}(x) \leq s(x)$
for all $x \in \pi_\tau^{-1}(D) \cap \delta$. 
Since there are only finitely  many pairs $D,\delta$ to check we can find the desired $m_\tau$. 

Firstly, if $D \cap \delta = \emptyset$,
then let 
\begin{align*}
\dist(D, \delta) := \min_{x \in \pi_\tau^{-1}(D) \cap \delta} \dist(x, H_\tau) > 0.
\end{align*} 
It suffices to choose $m \leq \frac{-c}{\dist(D,\delta)}$, where $c$ denotes 
$\max_{x\in \partial \sigma} s(x) - \min_{x\in \partial \sigma} s(x)$.

If $D \cap \delta \neq \emptyset$, then let 
\begin{align*}
\text{cone}(D,\delta) = \{ v \in \R^r \ | \ x + \varepsilon v \in \pi^{-1}_{\tau}(D) \cap \delta \text{ for some }  x \in D \cap \delta \text{ and some } \varepsilon >0\}, 
\end{align*}
 and take $v_1, \dots, v_l$ to be generators of this cone. 
Notice that the differentials 
$(\diff f_{m,\tau})_x (v_i)$ and $\diff s_y(v_i)$ are constant over all $x \in \pi^{-1}_{\tau}(D) $ and all  $y \in \delta$. 
Then choose an $m$ satisfying 
$(\diff f_{m,\tau})_x (v_i) \leq \diff s_y(v_i)$
for all $i$, all $x \in \pi^{-1}_{\tau}(D) $, and all $y \in \delta$. 
Such a choice of $m$ is  possible since the left hand side can be made arbitrarily small except 
for when $v_i$ lies in the lineality space of $\text{cone}(D,\delta)$.
In this case, both sides
agree since $f_\tau$ and $s$ agree on $D \cap \delta$.
By linearity it follows that 
\begin{align*}
(\diff f_{m,\tau})_x (v) \leq \diff s_y(v)
\end{align*}
for any $v \in \text{cone}(D,\delta)$, $x \in \pi^{-1}_{\tau}(D)$, and every $y \in \delta$. Finally, every 
$x \in \pi_\tau^{-1}(D) \cap \delta$ can be written in the form 
$x = x_0 + v$, where $x_0 \in D \cap \delta$ and $v \in \text{cone}(D,\delta)$. 
Then by choosing such an $m$, it  follows that for all $x \in \pi^{-1}_ \tau (D) \cap \delta $ we have $f_{m, \tau}(x) \leq s(x)$.

Now suppose that $\sigma$ is a polyhedron in $\TT^r$. 
We proceed by induction, with the base case being when all points in $\sigma$ are of sedentarity zero. 
In this case, the above argument applies. 
Now asume that the statement holds if  $\sigma$ does not intersect $\TT^r_i$ for $1 \leq i \leq   k-1$ but that $\sigma \cap \TT^r_k \neq \emptyset$. 
There  exists a  constant $c_k$ and a function $s_k \colon \TT^r_k \to \R$ such that $s|_{\sigma \cap H_k^- } = s_k \circ \pi_k$, where $H_k^-$ is the closed half-space defined by $\langle x, e_k \rangle \leq  c_k$. Let  $H_k$  be the hyperplane defined by $\langle x, e_k \rangle =  c_k$ and $H_k^+$ be the  closed half-space defined by $\langle x, e_k \rangle \geq  c_k$.

Now $\sigma \cap H_k^+$ is a rational polyhedron in $\TT^r$ which contains no points of sedentarity $\{k\}$. 
We define a rational function 
$s_k^- \colon \sigma \cap H_k^- \to \R$ given by $x \mapsto s_k(\pi_k(x))$.
Note that for all $x \in \partial \sigma \cap H_k^-$, we have $s_k^-(x) = s(x)$.
Hence, the rational functions $s$ and $s_k^-$ induce a function $s' \colon \partial (\sigma \cap H_k^+) \to \R$. 
By the induction assumption there exists a rational function $s_k^+$ on $\sigma \cap H_k^+$ which extends $s'$.
Then the function $s : \sigma \to \R$ given by 
\[
  x \mapsto 
		\begin{cases} 
		  s_k^+(x) & x \in \sigma \cap H_k^+, \\
		  s_k^-(x) & x \in \sigma \cap H_k^-, \\
		\end{cases}
\]
is a well-defined rational function that extends $s$, as required. 
\qed
\end{proof}

For any subset $K \subset X$  we define $$\mathcal{M}(K) \colon = \varinjlim_{\substack{U \subset X \text{open} \\ K \subset U}} \mathcal{M}(U).$$

\begin{lemma} \label{pc softness}
	Let $K$ be a compact polyhedral subset contained in an open star of $\CC$ let $s$ be a section in $\MS(K)$. 
	Then there exists $s' \in \MS(X)$ such that $s' |_K = s$.  
\end{lemma}

\begin{proof}
Note that since $s$ is rational function that is defined on an open neighborhood of $K$ 
it extends to a compact polyhedral subset $K_1$ that satisfies $K \subset \mathring K_1$. 
Let $K_2$ be a compact polyhedral subset such that $K_1 \subset \mathring K_2$. 
We can assume that $K_2$ is contained in the same open star as $K$.
We fix a polyhedral structure $\DS$ on $K_2$ such that $K_1$ is the 
support of a polyhedral subcomplex.  
We construct $s'$ inductively on the skeleta of $\DS$. 
Assume that an extension of $s$ is defined on the $k$-skeleton of $\DS$. 
Let $\sigma$ be a $k+1$ polyhedron in $\DS$. 
If $\sigma$ is contained in the boundary of $K_2$, then $s'|_{\sigma} = 0$.
If $\sigma \subset K_1$, then $s'|_{\sigma} = s|_{\sigma}$. 
Otherwise  Lemma \ref{lem:piecewiseaffineonsigma}  provides an extension of the rational function $s'|_{\partial \sigma}$ to  a rational function $s'|_{\sigma}$.   
Therefore, we can extend the rational function $s'$ to all of $K_2$. 

By construction the rational function $s'$ satisfies $s'|_{\partial K_2} = 0$.
By declaring $s'$ to be zero outside of $K_2$ we obtain a rational function on $X$.  
Since $s'|_{\mathring K_1} = s|_{\mathring K_1}$ we have $s'|_K = s$.  
\qed
\end{proof}

\begin{lemma}\label{prop:vanishingH1M}
If $X$ is a rational polyhedral space, then 
$\HH^{1}(X, \MS) = 0$. 
\end{lemma}

\begin{proof}
Take an injective map from  $\MS$ to an acyclic sheaf $\FS$ and  denote the quotient sheaf by $\GS := \FS / \MS$.  
We claim that $\FS(X) \to \GS(X)$ is surjective. 
For a  $t \in \GS(X)$, let $\US = (U_i)_{i \in I}$ be a cover of $X$ such that there exist a collection   $s_i \in \FS(U_i)$ that map to $t \vert_{U_i}$. 
Since $X$ is paracompact we may assume that $\US$ is locally finite. 
For each $U_i$ take  a locally finite cover by a collection of compact polyhedral subsets so that for each $i$ 
there is a member of the cover $K_i$ which is contained in $U_i$ and $(K_i)_{i \in I}$ still cover $X$.  
For $J \subset I$ set  $K_J \colon = \bigcup_{j \in J} K_j$. 
Since the covering of $X$ by the sets $K_i$ is locally finite, the union $K_J$ is closed.

The set
\begin{align*}
\ES := \{ (J, s_J) \mid J \subset I, s_J \in \FS(K_J) \text{ mapping to } t \vert_{K_J} \}
\end{align*}
carries a  partial order given by $(J, s) \leq (J', s')$ whenever $J \subset J'$ and $s'|_{K_{J}} = s$. 
By Zorn's lemma, $\ES$ has a maximal element $(J,s)$. 
We want to show $J=I$.

Assume that there exists $j \in I \setminus J$. 
Then $s_j - s$ maps to zero in $\GS(K_j \cap K_J)$
and hence is the image of an element $r \in \MS(K_j \cap K_J)$.
By Lemma \ref{pc softness}, we can extend $r$ to a section $r' \in \MS(X)$.
Let $s' \in \FS(X)$ denote the image of $r'$ and 
consider the section $s_j - s' \in \FS(K_j)$. 
By construction, this section agrees with $s$  on $K_{J} \cap K_j$. 
Therefore we can glue $s_j - s'$ and $s$ to a section of $\FS$ over $K_{J \cup j}$. But this is contradiction 
to the maximality of $(J,s)$.

Thus $\FS(X) \to \GS(X)$ is surjective. Using the long exact sequence in cohomology associated to
$0 \to \MS \to \FS \to \GS \to 0$ and the fact that $\FS$ is acyclic, we conclude that $\HH^1(X, \MS) = 0$. 
\qed
\end{proof}

\begin{proposition}\label{prop:sec}
Any line bundle $L \in \Pic(X)$ admits a section.
\end{proposition}

\begin{proof}
  Recall that the long exact sequence on cohomology associated to 
	the short exact sequence
	\begin{align*}
		0 \to \Aff \to \MS \to \MS / \Aff \to 0
	\end{align*}
	induces the maps
	\begin{align*} 
		\HH^{0}(X, \MS / \Aff) \stackrel{\delta}{\rightarrow} \HH^1(X, \Aff) \rightarrow \HH^1(X, \MS).
	\end{align*}
	By definition, we need to show that $\delta$ is surjective. 
	This follows from the vanishing of $\HH^1(X, \MS)$ established in Lemma
	\ref{prop:vanishingH1M}.
	\qed
\end{proof}

\subsection{Tropical spaces and the fundamental class}

Throughout this section $X$ is a rational polyhedral space of pure dimension $n$ with a polyhedral structure $\CC$ and 
$\DD$ is a $\CC$-stratified simplicial structure on $X$.

A point $x$ in a rational polyhedral subspace $Y$ of $\T^r$  is \emph{generic} if it 
admits an open neighbourhood in $Y$ which is an open set of an affine subspace of $\R^r_I$ for some $I$.
Being a  generic point is invariant under integral extended affine maps and hence this notion 
extends to rational polyhedral spaces.  
We denote the (open and dense) set of generic points by $X^{\gen} \subset X$.

\begin{defn}\label{def:weighted}
A rational polyhedral space  $X$ is \emph{weighted} if it is equipped with a locally constant function 
$\omega \colon X^{\gen} \to \Z\setminus \{0\}$. 
For a maximal face $\sigma \in \CC$, the function $\omega$ is constant on $\relint \sigma$ and 
we define $\omega(\sigma)$ to be this value. 
We also denote by $\omega(\Delta)$ its constant value on $\relint(\Delta) \in \DD_n$.
\end{defn}

We can extend Definition \ref{def:multitangent} to simplicial structures. For any $\Delta \in \DD_k$ whose relative interior is contained in $\sigma \in \CC$
we define  $\Linear(\Delta)$
to be the minimal 
linear subspace
 of $ \Linear(\sigma)$ which is defined over $\QQ$ and 
has the property that $\Delta$ is contained in a translate of $\Linear(\Delta)$.
Note that in general $k \leq \dim \Linear(\Delta)$. 
We set $\Linear_{\Z}(\Delta) := \Linear(\Delta) \cap \Linear_\Z(\sigma)$.
For $\Delta \in \DD_k$ with $\rank \Linear_{\Z}(\Delta) = k$, we 
define $\Lambda_\Delta$ to be the unique generator of 
$\bigwedge^k \Linear_{\Z}(\Delta) \cong \Z$ 
 compatible with the orientation of $\Delta$.
Then for a simplex $\Delta$, we define $\BF^1(\Delta)$ in the same way as for polyhedra in Definition \ref{def:multitangent}.

\begin{defn}[Fundamental chain] \label{def fundamental chain} \label{def:fundamentalclass}
The \emph{fundamental chain} of $X$ is 
\begin{align*}
  \ch(X) := \sum \limits_{\Delta \in \DD_n} \omega(\Delta) \Lambda_\Delta \otimes \Delta \in C^{\BM}_{n,n}(X, \ZZ).
\end{align*}

We call $X$ an \emph{(abstract) tropical space} if $\ch(X)$ is closed. 
In this situation, we call $[X] := [\ch(X)] \in \HH_{n,n}^{\BM}(X)$ the \emph{fundamental class} of $X$. 
\end{defn}

When $X$ is a tropical space, it is straightforward to check that the class $[X]$  does not depend on the choice of  simplicial structure  on $X$.

\begin{rem}{\rm
 The more conventional definition of a tropical space refers to the so-called balancing condition \cite[Definition 3.3.1]{MacStu},
\cite[Section 6.1]{MikRau}. 
To formulate this condition in our context, first let us use 	the notation $\Delta' \face \Delta$ to indicate pairs
$\DD_{n-1} \ni \Delta' \prec \Delta \in \DD_n$ of the  \emph{same sedentarity}.
Let  $\Delta' \in \DD_{n-1}$ be  such that 
$\Linear_{\Z}(\Delta')$ has rank $n-1$. 
A \emph{primitive generator} of a  pair $\Delta' \face \Delta$
is an integer vector $v_{\Delta, \Delta'}$ such that 
\begin{align} \label{eq primitive vector} 
v_{\Delta, \Delta'} \wedge \Lambda_{\Delta'} = \varepsilon_{\Delta, \Delta'} \Lambda_\Delta,
\end{align} 
where $\varepsilon_{\Delta, \Delta'}$ is the sign with which $\Delta'$ appears in $\partial \Delta$.
Primitive generators are unique up to adding an element in $\Linear_{\Z}(\Delta')$.
The rational polyhedral space $X$ is called \emph{balanced at $\Delta'$} if
\begin{align} \label{eq balancing condition} 
\sum_{\Delta : \Delta' \face \Delta} \omega(\Delta) v_{\Delta, \Delta'} \in \Linear_{\Z}(\Delta'),
\end{align} 
where the $v_{\Delta, \Delta'}$ are primitive generators.
The space 	$X$ is called \emph{balanced} if it is balanced at all  $\Delta' \in \DD_{n-1}$ such that $\Linear_{\Z} (\Delta')$ has rank $n-1$.
It follows from \cite[Proposition 4.3]{MikZhar} that $\ch(X)$ is a  closed $(n,n)$-cycle  if and only if $X$ is balanced. 
In particular, whether or not $X$ is a tropical space does not depend on the choice of  simplicial structure $\DD$. }
\end{rem}

\begin{defn}\label{def:contract}
Given $l \in {\bf F}_\Z^p(\sigma)$ and $v \in {\bf F}^\Z_{p'}(\sigma)$ with $p \leq p'$, 
the \emph{contraction} $\langle l; v \rangle \in {\bf F}^\Z_{p'-p}(\sigma)$
is induced by the usual contraction map 
$\langle \; ; \; \rangle \colon \bigwedge^p (\Z^r_I)^* \times \bigwedge^{p'} \Z^r_I \to \bigwedge^{p'-p} \Z^r_I$.
More generally, given $\tau, \tau' \prec \sigma$ and $l \in {\bf F}_\Z^p(\tau), v \in {\bf F}^\Z_{p'}(\sigma)$, the \emph{contraction} 
$\langle l; v \rangle$ is given by 
\begin{align*}
\langle l; v \rangle := i_{\tau',\sigma}(\langle \rmap(l); v \rangle) \; \in {\bf F}^\Z_{p'-p}(\tau').
\end{align*}
\end{defn}

\begin{defn}\label{def:capwithfun}
The  \emph{cap product with the fundamental class of $X$} is the map
\begin{align*}
\cap [X] \colon C^{p,q}(X, \Z)  &\to C_{n-p,n-q}^{\BM}(X, \Z) \\
\alpha & \mapsto \sum_{[i_0, \dots, i_n] \in \DD_n} \omega(\Delta) \langle \alpha([i_0, \dots, i_q]); \Lambda_\Delta \rangle \otimes [i_q, \dots, i_n], 
\end{align*}
where $\langle \; ; \; \rangle$ denotes the contraction introduced in Definition \ref{def:contract}.\end{defn}

The definition of the cap product can be extended to $\cap \sigma$ for arbitrary simplicial chains $\sigma \in C^{\BM}_{p',q'}(X, \ZZ)$.
Also the Leibniz formula $d(\alpha \cap \sigma) = (-1)^{q+1} (\delta \alpha \cap \sigma - \alpha \cap d\sigma)$ holds 
on the chain level. 
If $X$ is a tropical space then $d\ch(X) = 0$, and it follows that the map $\cap [X]$ described above descends to 
to a map between cohomology and Borel-Moore  homology groups 
\begin{align*}
\cap [X] \colon \HH^{p,q}(X, \Z)  &\to \HH_{n-p,n-q}^{\BM}(X, \Z). 
\end{align*}
To see that $\cap [X]$ does not depend on the simplicial structure on $X$, 
note that the cap product can also be described on the level of singular chains.

\subsection{Subspaces, Divisors, and the Chern class map}

Throughout this section $X$ is a rational polyhedral space of pure dimension $n$ with a polyhedral structure~$\CC$.

\begin{defn}\label{def:cycle}
A subset $Z \subset X$ is a \emph{rational polyhedral subspace} if it is closed and 
the restrictions of the charts of an atlas of $X$ provide an atlas as a rational polyhedral space for $Z$.
\end{defn}

Given a rational polyhedral subspace $Z$, there exists a $\CC$-stratified simplicial structure $\DD$ 
for $X$ such that $Z$ a union of cells of $\DD$. We call such a structure fine enough for $Z$. 
It can be constructed inductively as a simplicial structure of $|\CC_k|$.
Indeed, for each $\sigma \in \CC_k$, the intersection $Z \cap \sigma$ is a closed polyhedral subset
of the polyhedron $\sigma$ and any simplicial structure on $\partial \sigma$ fine enough for $Z \cap \partial \sigma$
can be extended to a simplicial structure on $\sigma$ fine enough for $Z \cap \sigma$.

Assume further that $Z$ is of pure dimension $k$ and is equipped with a weight function $\omega$.
We define 
\begin{align*}
  \ch(Z) := \sum \limits_{\substack{\Delta \in \DD_{k}  \\  \Delta \subset Z} } \omega(\Delta) \Lambda_\Delta \otimes \Delta \in C^{\BM}_{k,k}(X, \ZZ).
\end{align*}

\begin{defn}
The weighted subspace $Z$ is called a \emph{tropical cycle} if $\ch(Z)$ is a closed chain. In this  case, we 
denote $\cyc(Z) \in \HH_{k,k}^{\BM}(X, \Z)$ the \emph{cycle class} of $Z$.
\end{defn}

The tropical cycles of dimension $k$ 
form a group under taking unions and adding up weights
(see \cite[Lemmas 2.14 and 5.15]{AllermannRau} and \cite[Section 7.1]{MikRau}), 
which we denote by $\ZY_k(X)$.
With these definitions, the map
\begin{align*}
\cyc \colon \ZY_k(X) \to \HH_{k,k}^{\BM}(X, \Z);  \quad Z \mapsto \cyc(Z) 
\end{align*}
is a homomorphism. 

We are interested in a construction
which produces a  tropical cycle of dimension $n-1$ from  a Cartier divisor.
Let $s$ be a section of a line bundle $L \in \Pic(X)$ and  consider the subset of $X$ given by
\begin{align*}
\ND(s) :=  \{x \ | \   s_i \text{ is not affine in a neighbourhood of $x$}\}.
\end{align*}
It is a rational polyhedral subspace of dimension $n-1$.
Next we define weights on this set to turn it into a tropical cycle following \cite[Section 3]{AllermannRau} and \cite[Section 5.2]{MikRau}.

\begin{defn}\label{def:divcyc}
The \emph{divisor map} is 
\[
  \div \colon \Div(X)  \to \ZY_{n-1}(X), 
\]
where $\div(s) $ is a tropical codimension one cycle supported on $\ND(s)$.
The weight of a generic point $x \in \ND(s)^{\gen}$ 
(which is of sedentarity zero) is given as follows.
Fix $i$ such that $x \in U_i$ and choose 
\begin{itemize}
\item a neighbourbood $x \in U \subset U_i$ and simplicial structure $\DD$ on $U$ such that
$\ND(s) \cap U \subset |\DD_{n-1}|$ and $x \in \relint \Delta', \Delta' \in \DD_{n-1}$,
\item primitive generators $v_{\Delta, \Delta'}$ for any pair $\Delta' \face \Delta$,
\end{itemize}
and set
\begin{align} \label{eq weights of divisor}
\omega_{\div(s)}(x) = \sum_{\Delta : \Delta' \face \Delta }  
\omega_X(\Delta) \diff s_i|_{\Delta}(v_{\Delta, \Delta'}) - \diff s_i|_{\Delta'}(\sum_{\Delta : \Delta' \face \Delta } \omega_X(\Delta) v_{\Delta, \Delta'}).
\end{align}
\end{defn}

The weights $\omega_{\div(s)}(x)$ may happen to be zero;
such parts of $\ND(s)$ are tacitly removed from $\div(s)$. 
For details on this construction and proof of its well-definedness we refer to 
\cite[Sections 3 and 6]{AllermannRau} as well as \cite[Section 5.2]{MikRau}.

\begin{thm}\label{thm:commutingdiagram}
The diagram
\begin{align}\label{diag:comm}
\begin{xy}
\xymatrix{
\Div(X) \ar[r]   \ar[d]_{\div}  & \Pic(X) \ar[r]^{c_1}  & \HH^{1,1}(X, \Z) \ar[d]^{\cap [X]} \\
\ZY_{n-1}(X) \ar[rr]^{\cyc}                                                            &&  \HH^{\BM}_{n-1,n-1}(X, \Z)
}
\end{xy}
\end{align}
commutes.
\end{thm}

\begin{lemma}\label{lem:cycdiv}
Let $s = (s_i)_i \in \Div(X)$ be a Cartier divisor and let $\DD$ be a simplicial structure fine enough for $\ND(s)$.
Then 
\begin{align}\label{eqn:cycdivcell} 
\cyc (\div (s)) =   \sum_{\substack{\Delta \in \DD_n \\ \Delta' \in \DD_{n-1} \\ \Delta' \prec \Delta}}
\omega(\Delta) \cdot 
\imapdd\left(\langle \diff s_{i_{\Delta'}}|_{\Delta}; \varepsilon_{\Delta, \Delta'}  \Lambda_\Delta \rangle\right)  \otimes \Delta' 
\in \HH^{\BM}_{n-1, n-1}(X), 
\end{align}
where $\varepsilon_{\Delta, \Delta'} = \pm 1$ is the relative orientation of $\Delta$ and $\Delta'$ and $i_{\Delta'}$ is chosen 
so that $\relint \Delta' \subset U_{i_{\Delta'}}$.
\end{lemma}

\begin{proof}
Since $\DD$ is a simplicial structure on $X$ fine enough for $\ND(s)$, the simplices $\{\Delta \in \DD | \Delta \subset \ND(s)\}$
form a simplicial structure for $\ND(s)$. 
In particular, it follows that for any $\Delta' \in \DD_{n-1}$ with 
$\Delta' \subset \ND(s)$ the lattice $\Linear_{\Z}(\Delta')$ is of rank $n-1$. 
For any $\Delta' \face \Delta$ we choose primitive generators $v_{\Delta, \Delta'}$ (satisfying
\eqref{eq primitive vector}) and use the notation $v_{\Delta'} \in \Linear_{\Z}(\Delta')$ for the balanced sum in \eqref{eq balancing condition}.
We now want to show that for any $\Delta' \in \D_{n-1}$ its coefficients 
on the right hand and left hand side of \eqref{eqn:cycdivcell} agree.

First assume $\sed(\Delta') = 0$ and $\Delta' \subset \ND(s)$.
Then for any facet $\Delta' \prec \Delta$ 
the rules of contracting wedge products along $1$-forms 
applied to equation \eqref{eq primitive vector} provide
\begin{align}\label{eqn:contract}
\langle \diff s_{i_{\Delta'}}|_{\Delta}; \varepsilon_{\Delta, \Delta'} \Lambda_{\Delta} \rangle = 
\diff s_{i_{\Delta'}}|_{\Delta}(v_{\Delta, \Delta'}) \cdot \Lambda_{\Delta'} - 
 v_{\Delta, \Delta'} \wedge \langle \diff s_{i_{\Delta'}}|_{\Delta}; \Lambda_{\Delta'} \rangle.
\end{align}
Moreover, since $v_{\Delta'} \wedge \Lambda_{\Delta'} = 0$, we have
\[
  \diff s_{i_{\Delta'}}|_{\Delta}(v_{\Delta'}) \cdot \Lambda_{\Delta'} = 
v_{\Delta'} \wedge \langle \diff s_{i_{\Delta'}}|_{\Delta}; \Lambda_{\Delta'} \rangle.
\]
Comparing with \eqref{eq weights of divisor}, this proves the equality of coefficients.

Let us now consider $\Delta' \not\subset \ND(s)$. 
Then $s_{i_{\Delta'}}$ is affine linear in a neighbourhood of $\relint(\Delta')$ and 
therefore $\diff s_{i_{\Delta'}}$ defines an element
in ${\bf F}^1(\Delta')$ (independent of a choice of facet $\Delta' \prec \Delta$). 
The coefficient of $\Delta'$ in \eqref{eqn:cycdivcell} is therefore equal to 
\begin{align*}  
\sum_{\substack{\Delta \in \DD_{n} \\ \Delta' \prec \Delta}}
\omega(\Delta)  \langle \diff s_{i_{\Delta'}}|_{\Delta}; \varepsilon_{\Delta, \Delta'} \Lambda_\Delta \rangle = 
\langle \diff s_{i_{\Delta'}}; \sum \omega(\Delta) \varepsilon_{\Delta, \Delta'} \Lambda_\Delta \rangle.
\end{align*}
But the sum on the right hand side is exactly equal to the coefficient of $\Delta'$ in $\partial \ch(X)$, so 
it is zero since $\ch(X)$ is closed. \qed
\end{proof}

\begin{proof}[Proof of Theorem \ref{thm:commutingdiagram}]
Let $s$ be a section of a line bundle $L \in \Pic(X)$. 
Let  $\DD$ be a simplicial structure on $X$ fine enough for $\ND(s)$.
We can assume that each open star of $\DD$ is fully contained in the domain $U_i$ of $s_i$
for some $i$. Hence, by fixing an appropriate choice and restricting to open stars, 
we can even assume that $s = (s_i)_{i\in \DD_0}$ is labelled by the vertices of $\DD$ and
that $U_i$ is equal to the open star around the vertex $i$. As usual we use the notation
$U_{i_0 \dots i_n} = U_{i_0} \cap \dots \cap U_{i_n}$ for the open stars of higher-dimensional simplices
$\Delta = [i_0, \dots, i_n] \in \DD$.

We first compute the image of  $s$  following the upper right path.
Since $s$ is a section of $L$, the transition functions for $L$ are given
by $f_{ij} = s_i - s_j$ on $U_{ij}$. Using the identification 
of \v{C}ech cochains and simplicial cochains explained in Remark \ref{rem:Cech vs simplicial}, 
we conclude that $c_1(L)$ is the simplicial $(1,1)$-cochain which, when applied to an edge $[i,j] \in \DD_1$,
provides $\diff f_{ij} \in {\bf F}^1([i,j])$. 
Capping with the fundamental class then gives
\begin{align} \label{lower left}
c_1(L) \cap [X] = \sum_{\Delta = [i_0,\dots,i_n]} \omega_\Delta  
\iota_{\Delta, \Delta_0} \langle  d f_{i_0 i_1}; \Lambda_\Delta \rangle \otimes [i_1,\dots,i_n],
\end{align}
where $\Delta_j := [i_0, \ldots,  \hat{i}_j, \ldots, i_n]$.
Let us now compute the effect of the lower left path using Lemma \ref{lem:cycdiv}. 
To do so, we fix the choice of indices $i_{\Delta'}$ required in  Lemma \ref{lem:cycdiv} by setting
$i_{\Delta'} = i_1$ for any $\Delta' = [i_1, \dots, i_n] \in \DD_{n-1}$.
In other words, to compute the coefficient of $\Delta'$ we always use the function associated to 
the \textit{first} vertex in $\Delta'$.
Let us now fix a maximal simplex $\Delta = [i_0, \dots, i_n] \in \DD_{n}$.
Then by Lemma \ref{lem:cycdiv} and with the convention just made, the contribution of  
 $\Delta$ to  $\cyc(\div(s))$ in \eqref{eqn:cycdivcell} is 
\begin{align*}
 \omega_\Delta  \left(  \iota_{\Delta, \Delta_0} \langle d s_{i_1}; \Lambda_{\Delta} \rangle \otimes  [i_1, \ldots, i_n] +  
\sum_{j=1}^n  \iota_{\Delta, \Delta_j} \langle d s_{i_0}; \varepsilon_{\Delta, \Delta_j} \Lambda_\Delta \rangle 
\otimes  [i_0, \ldots,  \hat{i}_j, \ldots, i_n]\right).
\end{align*}
Since the section $s$ satisfies 
$s_{i_1} = f_{i_0 i_1} + s_{i_0}$ 
we obtain
 \begin{align*} \label{eq homologous}
\omega_\Delta  \left( \iota_{\Delta, \Delta_0} \langle d f_{i_0 i_1}; \Lambda_\Delta \rangle \otimes  [i_1,  \ldots, i_n] +   
\sum_{j=0}^n \varepsilon_{\Delta, \Delta_j} 
\iota_{\Delta, \Delta_j} \langle  d s_{i_0}; \Lambda_\Delta \rangle \otimes  [i_0, \ldots, \hat{i}_j, \ldots, i_n] \right).
\end{align*}
The first summand is precisely the sum which appears in Equation (\ref{lower left})
and the second summand is homologous to zero (namely, equal to $\partial( \langle d s_{i_0}; \Lambda_\Delta \rangle \otimes \Delta)$). 
Therefore the diagram is commutative. 
\qed
\end{proof}

\begin{proof} [Proof of Theorem \ref{Theorem II intro}]

Suppose that $\alpha \in  \Ker(\phi \colon \HH^{1, 1}(X, \Z) \to \HH^{0,2}(X, \R))$.
By Theorem  \ref{Theorem I intro} there exists $L$ in $\Pic(X)$ 
such that $c_1(L) = \alpha$. 
By Proposition \ref{prop:sec} there exists a section $s$ of $L$.
By the commutativity of the diagram in Theorem \ref{thm:commutingdiagram}, we have 
$\alpha \cap [X] = c_1(L) \cap [X] = \cyc(\div(s))$. 
The statement of the theorem now follows since $\cyc(\div(s))$ is the fundamental class of a codimension one tropical cycle. 
 \qed
\end{proof}

\section{Poincar\'e duality with coefficients in $\Z$}
\label{sec:poincare}

In this section we prove Theorem \ref{Theorem III intro}. 
To do so, we restrict to tropical manifolds and establish a version of Poincar\'e duality over $\Z$.  
We first introduce tropical manifolds which are tropical spaces locally modelled on matroidal fans. 
We do not describe these fans here but refer the reader to the literature for their definition and properties \cite{ArdilaKlivans,Shaw:IntMat,FrancoisRau}.

\begin{defn} 
\label{def:tropicalmanifold}
 A \emph{tropical manifold} $X$ is a tropical space  whose weight function is equal to one and 
which is equipped with an atlas of charts 
$\{\varphi_{\alpha} \colon U_{\alpha} \rightarrow \Omega_{\alpha} \subset X_{\alpha}\}_{{\alpha}  \in A}$ such that  
$\Omega_{\alpha} \subset \T^{s_{\alpha}} \times \R^{r-s_{\alpha}}$ and 
$X_{\alpha} \cap (\T^{s_{\alpha}} \times \R^{r-s_{\alpha}}) = \T^{s_{\alpha}} \times X'_{\alpha}$, where 
$X'_{\alpha} \subset \R^{r-s_{\alpha}}$  is the support of a matroidal fan. 
\end{defn}

\begin{defn}
A tropical space $X$ satisfies \emph{Poincar\'e duality with integral coefficients} if   
\[
 \cap [X] \colon \HH^{p,q} (X, \Z) \to \HH_{n-p,n-q}^{\BM}(X, \Z) 
\]
is an isomorphism for all $p,q$. 
\end{defn}

\begin{thm}\label{thm:PD}
A tropical manifold satisfies Poincar\'e duality with integral coefficients. 
\end{thm}

The above theorem is a extension  of the version of Poincar\'e duality with real coefficients for tropical manifolds previously proved in \cite{JSS}. 
This version related tropical cohomology and tropical cohomology with compact support with real coeffcients
via a pairing given by integration. 

We first prove this version of Poincar\'e duality  for matroidal fans using a cellular description of tropical (co)homology. 
Let $X$ be a polyhedral subspace in $\T^r \times \R^s$ and $\CC$ a polyhedral structure on $X$ 
such that every cell contains a vertex. 
Then there are descriptions of tropical (Borel-Moore) homology and cohomology in terms of cellular chain complexes with respect to $\CC$, 
\begin{align}
\label{cellular Borel Moore} &\HH_{p,q}^{\BM}(X, \Z) =\HH_{q}(C_{p,\bullet}^{\BM, \cell}(\CC, \Z)), 
\text{ where } C_{p,q}^{\BM, \cell}(\CC, \Z) = \bigoplus_{\sigma \in \CC_q} {\bf F}_p^\Z(\sigma), \\
\label{cellular homology} &\HH_{p,q}(X, \Z)  =  \HH_{q}(C_{p,\bullet}^{\cell}(\CC, \Z)), 
\text{ where } C_{p,q}^{\cell}(\CC, \Z) = \bigoplus_{\substack{\sigma \in \CC_q \\ \sigma \text{ \tiny compact}}} {\bf F}_p^\Z(\sigma)  \text{ and} \\
&\label{cellular cohomology} \HH^{p,q}(X, \Z) = \HH^{q}(C^{p,\bullet}_{\cell}(\CC, \Z)), 
\text{ where } C^{p,q}_{\cell}(\CC, \Z) = \bigoplus_{\substack{\sigma \in \CC_q \\ \sigma \text{ \tiny compact}}} {\bf F}^p_\Z(\sigma). 
\end{align}
For a justification of these identifications see Remark \ref{simplicialhomology}.

Let $V $ be a fan in $\R^s$ and $\CC$ a polyhedral fan structure on $V$ such that $0$ is a vertex of $\CC$. 
We then find using (\ref{cellular homology}) and (\ref{cellular cohomology}) that 
$\HH^{p,0}(V, \Z) = {\bf F}^p_\Z(0)=:{\bf F}^p_\Z(V)$, $\HH_{p,0}(V) = {\bf F}_p^\Z(0)=:{\bf F}_p^\Z(V)$, 
$\HH^{p,q}(V, \Z) = \HH_{p,q}(V, \Z) = 0$ for all $p\geq 0, q > 0$.

Let $X$ be a tropical variety in $\R^r$ and $f \in \MS(X)$ be a tropical rational function. 
The graph $\Gamma_X(f)$ of $f$ is a polyhedral complex in $\R^{r+1}$. 
For every face $\tau$ of $\div(f)$ denote by $\tau_\leq$ the polyhedron 
$\{ (x,y) \vert x \in \tau, y \leq f(x)\} \subset \R^{r+1}$. 
The union $Y := \Gamma_X(f) \cup \{ \tau_\leq \vert \tau \in \div(f) \}$ is a tropical cycle in $\R^r$ if we define
the weights on $\Gamma_X(f)$ to be inherited from $X$ and the weight on a face  $\tau_\leq$ is defined to be  equal to the weight of  $\tau$ in $\div(f)$. 

\begin{defn}
We call $Y$ the \emph{open tropical modification of $X$ along $f$} and 
$\delta \colon Y \to X$ an \emph{open tropical modification}. 
The space $\div(f)$ is called the \emph{divisor} of the modification. 
\end{defn}

More details on this construction can be found in \cite[Construction 3.3]{AllermannRau} as well as \cite[Chapter 5]{MikRau}.
We provide some notation useful for modifications.  
Let $\delta \colon V \to W$ be an open tropical modification of fans along a function $f \in \MS(W)$ with divisor $D$ and 
let $\CC$ be a polyhedral structure consisting of cones which contains a polyhedral structure for $D$. 
We always assume that $\delta$ is induced by the projection $\pi \colon \R^r \times \R \to \R^r$
with kernel $e_{r+1}$. 
Denote by $\overline{V}$ the closure of $V$ in $\R^r \times \T$. 
Then the polyhedra
\begin{align} \label{modification structure}
  \lift{\sigma} &= (\id \times f)(\sigma) & & \text{for all } \sigma \in \CC \nonumber \\
  \down{\sigma} &= \lift{\sigma} + (\{0\} \times [-\infty, 0]) & & \text{for all } \sigma \in \CC, \sigma \subset D  \\  
  \bdry{\sigma} &= \sigma \times \{-\infty\} & & \text{for all } \sigma \in \CC, \sigma \subset D \nonumber 
\end{align}
form a polyhedral structure on $\overline{V}$. The intersection of the first two types of cones with $V$  form a polyhedral structure 
on $V$.

\begin{prop}\label{prop:pdmatZ2}
Let $V$ be a  matroidal fan  in $\R^s$.
Then $V$ satisfies Poincar\'e duality with integral coefficients. 
\end{prop}

\begin{proof}
We use induction on  $s$, with base case being when $s= \dim(V)$. 
In this case the support of $V$ is $\R^s$ and the statement can be verified using the fact that 
$\mathbf{F}^{\Z}_p$ and $\mathbf{F}_{\Z}^p$ are constant for all $p$.

If $s > \dim(V)$, by  \cite[Proposition 2.2]{Shaw:IntMat} there exists an open tropical modification $\delta \colon V \to W$
with divisor $D$ where $W$ and $D$ are matroidal fans in $\R^{s-1}$.  
The closure $\overline{V}$ of $V$ in $\R^{s-1} \times \T$ satisfies 
$\HH^{p,q}(W, \Z) = \HH^{p,q}(\overline{V}, \Z)$ and $\HH_{p,q}^{\BM}(W, \Z) = \HH_{p,q}^{\BM}(\overline{V}, \Z)$ 
by Proposition  \ref{prop:modarg}.
We have the short exact sequence   
\begin{align} \label{ses:BM}
0 \rightarrow C_{p,q}^{\BM, \cell}(D, \Z)
\rightarrow C_{p,q}^{\BM, \cell}(\overline{V}, \Z) \rightarrow C_{p,q}^{\BM, \cell}(V, \Z) \rightarrow  0 .
\end{align}
Using the notation from (\ref{modification structure}), 
the maps in \eqref{ses:BM} are given by 
\begin{align*} 
    &                                                  & & & v \otimes \lift{\sigma} &\mapsto v \otimes \lift{\sigma}, \\
    v \otimes \sigma &\mapsto v \otimes \bdry{\sigma}, &\text{and} & & v \otimes \down{\sigma} &\mapsto v \otimes \down{\sigma}, \\
    &                                                  & & & v \otimes \bdry{\sigma} &\mapsto 0. 
\end{align*}
 The same induction argument as in \cite[Lemma 4.26]{JSS} proves that 
$\HH^{\BM}_{p, q}(V, \Z) = 0$ for  $q \neq n$ and that the long exact sequence obtained from (\ref{ses:BM}) reduces to 
a short exact sequence which fits into the following  commutative diagram (see (\ref{Orlik Salomon sequence}))
\begin{align}\label{commdiagPD}
\begin{xy}
\xymatrix{
0 \ar[r] & \HH^{p, 0}(W, \Z) \ar[r] \ar[d]^{\cap [W]} & {\HH}^{p, 0}(V, \Z) \ar[r] \ar[d]^{\cap [V]} & 
{\HH}^{p-1, 0}(D, \Z) \ar[d]^{\cap [D]} \ar[r] & 0 \\ 
0 \ar[r] & \HH_{n-p,n}^{\BM}(W, \Z) \ar[r] & \HH_{n-p,n}^{\BM}(V, \Z) \ar[r] & \HH_{n-p,n-1}^{\BM}(D, \Z) \ar[r] & 0.
}
\end{xy}
\end{align}
Arguing by induction we can assume that  $W$ and $D$ satisfy Poincar\'e duality with integral coefficients. 
Then the  five lemma shows that $V$ does as well. This completes the proof. 
\qed
\end{proof}

\begin{prop}\label{prop:modarg}
Let $\delta \colon V \to W$ be an open tropical modification with divisor $D$, such that $W$, $V$ and $D$ are matroidal fans. 
Then $\delta_* \colon \HH_{p,q}^{\BM}(\overline{V}, \Z) \to \HH_{p,q}^{\BM}(W, \Z)$
and $\delta^* \colon \HH^{p,q}(W, \Z) \to \HH^{p,q}(\overline{V}, \Z)$ are isomorphisms for all $p,q$.
\end{prop}

\begin{proof}
For cohomology, note that $\overline{V}$ has three compact cells, 
which we denote by $\tau_\infty$, $\tau_0$ and $\tau_\leq$. 
We further have $\BF^p_{\Z}(\tau_\infty) = \BF_{\Z}^p(D)$, $\BF^p_{\Z}(\tau_0) = \BF^p_{\Z}(V)$ and 
$\BF^p_{\Z}(\tau_\leq) = \BF^p_{\Z}(D) \oplus (\BF^{p-1}_{\Z}(D) \wedge w)$,
where $w \colon \R^n \to \R$ is any $\Z$-linear form such that $w(e_{r+1}) = 1$.  
By (\ref{cellular Borel Moore}) we thus have to show that the cohomology of
\begin{align*}
\begin{xy}
\xymatrix{
0 \ar[r] &  \BF^p_\Z(D) \oplus \BF^p_\Z(V) \ar[r] & \BF^p_\Z(D) \oplus ( \BF^{p-1}_\Z(D) \wedge w ) \ar[r] & 0
}
\end{xy}
\end{align*}
is equal to $\BF^p_\Z(W)$. 
This follows from dualising the sequence (\ref{Orlik Salomon sequence}).

It remains to prove the statement about Borel-Moore homology. 
In the following, we always use $\sigma$ to denote arbitrary cells of $W$ and
$\tau$ for cells of codimension at least $1$. 
On the chain level 
$\delta_*$ is given by the map $\Psi \colon C_{p,q}^{\BM}(\clX, Q) \to C_{p,q}^{\BM}(W, Q)$
defined by $v \otimes \lift{\sigma} \mapsto \pi(v) \otimes \sigma$, 
$v \otimes \down{\tau} \mapsto 0$,
$v \otimes \bdry{\tau} \mapsto v \otimes \tau$.
We want to show that this is a quasi-isomorphism.

\emph{Injectivity:} 
Let $C$ be a closed chain in $C_{p,q}^{\text{BM}}(\clX, Q)$ with $\Psi(C) = \partial B$ being a boundary. 
By choosing an arbitrary lift $\lift{B}$ of $B$ under $\Psi$, which is surjective, 
and subtracting $\partial \lift{B}$, we can assume $\Psi(C) = 0$. 
Note that the cells $\bdry{\tau}$ can be moved in the interior by adding a suitable boundary of $\down{\tau}$ cells, 
and hence we can assume $C = \sum a_\sigma \lift{\sigma} + \sum b_\tau \down{\tau}$. 
Then $\Psi(C) = 0$ implies that $a_\sigma \in \ker(\pi)$. From Lemma \ref{lem:coeffmodifacation}, it follows that if $q = n$ then  $a_\sigma = 0$,  
and if $q < n$ then $a_\sigma \in {\bf F}_p^\Z(\down{\sigma})$.
Hence, by adding the boundaries of $a_\sigma \down{\sigma}$, we can assume $C = \sum b_\tau \down{\tau}$.
But then $0 = \partial C = \sum b_\tau \lift{\tau} + \sum c_\rho \down{\rho} - \sum \pi(b_\tau) \bdry{\tau}$ implies $b_\tau = 0$.

\emph{Surjectivity:} 
 Let $C = \sum a_\sigma \sigma$ be a closed chain in $C_{p,q}^{\BM}(W, Q)$. 
We can obviously find a lift of $C$ of the form $C_1 := \sum \lift{a}_\sigma \lift{\sigma}$. 
Its boundary is of the form $\partial C_1 = \sum b_\tau \lift{\tau}$.
Since $\Psi(\partial C_1) = 0$, we get $\pi(b_\tau) = 0$ and hence
$b_\tau \in {\bf F}_p^\Z(\down{\tau})$ by Lemma \ref{lem:coeffmodifacation}.
Hence by adding $b_\tau \down{\tau}$, we obtain $C_2$ with $\partial C_2 = \sum_\rho c_\rho \down{\rho}$,
where $\rho$ runs through cells of dimension $q-2$. 
Let us compute $c_\rho$. 
By construction it is a sum over the flags $\rho \subset \tau \subset \sigma$, each contributing $\pm \lift{a}_\sigma$. 
But for each $\sigma$ there are exactly two such flags, and they contribute with
opposite sign, which implies $c_\rho = 0$. 
Hence we get $\Psi(C_2) = \Psi(C_1) = C$ and $\partial C_2 = 0$, as required.
\qed
\end{proof}

\begin{lemma} \label{lem:coeffmodifacation}
  Let $v \in {\bf F}_p^\Z(\lift{\tau})$ such that $\pi(v)=0$. Then $v \in {\bf F}_p^\Z(\down{\tau})$.
\end{lemma}

\begin{proof}
The statement follows from the fact that the sequence  
\begin{align} \label{Orlik Salomon sequence}
 \begin{xy}
 \xymatrix{
0 \ar[r] & {\bf F}_{p-1}^\Z(D)  \ar[rr]^{w \mapsto w \wedge e_{r+1}} && 
{\bf F}_{p}^\Z(V) \ar[rr]^{v \mapsto \pi(v)} && {\bf F}_{p}^\Z(W) \ar[r] & 0 
}
\end{xy}
\end{align}
is exact. 
This sequence is obtained by combining a similar short exact sequence for 
Orlik-Solomon algebras from \cite[Theorem 3.65]{OrlikTerao} together with the relation between the Orlik-Solomon algebras
and ${\bf F}_p$ from  \cite{Zharkov:Orlik}.					   
\qed
\end{proof}

\begin{lemma} \label{lemma product with T}
Let\,$Y$\,be a polyhedral space in $\T^s \times \R^{r-s}$.\,Then $\HH_{p,q}^{\BM}(Y, \Z) =\HH_{p+1,q+1}^{\BM}( Y \times \T, \Z)$ 
and $\HH^{p,q}(Y, \Z) = \HH^{p,q}(Y \times \T, \Z)$ for all $p,q$.
\end{lemma}

\begin{proof}
Let $\CC$ be a polyhedral structure on $Y$. 
Given a face $\sigma \in \CC$, denote $\sigma_\infty := \sigma \times \{-\infty\}$ and $\widetilde{\sigma} := \sigma \times \T$.
The collection of all these polyhedra forms a polyhedral structure on $Y \times \T$. 
The statement for cohomology now follows directly from (\ref{cellular cohomology}) since the 
compact cells for the polyhedral structure on $Y \times \T$ are precisely of the form $\sigma \times \{- \infty\}$ for 
compact cells $\sigma$ of $\CC$.

For Borel-Moore homology, 
we prove the claimed isomorphism  by constructing
an explicit homotopy equivalence on the cellular chain complexes with respect to these polyhedral structures. 
Let us first look at the behaviour of the multi-tangent spaces. There are projection and lifting maps 
\begin{align*}
\pi  \colon {\bf F}_p(\widetilde{\sigma}) \to {\bf F}_p(\sigma_\infty) = {\bf F}_p(\sigma) 
\text{ and }
 \wedge e_{r+1} \colon {\bf F}_p(\sigma) \to {\bf F}_{p+1}(\widetilde{\sigma}). 
\end{align*}
The map $\pi$ is induced by the linear projection $\R^r \times \R \to \R^r$ forgetting the last coordinate. 
The second map is given by   $v\mapsto v \wedge e_{r+1} := \widetilde{v} \wedge e_{r+1}$ where $\widetilde{v} \in \pi^{-1}(v)$
and $e_{r+1}$ denotes the kernel of the map $\pi$.  
Note that the wedge product does not depend on the choice of preimage. 
Let $w \colon \R^r \times \R \to \R$ be  the linear form  
given by projecting onto the last factor,
regarded as an element $w \in  {\bf F}^1(\sigma)$, as in the proof of Proposition \ref{prop:modarg}. 
We define	
\begin{align*} 
\Psi \colon C_{p,q}^{\BM, \cell}(Y,\Z) \to C_{p+1,q+1}^{\BM, \cell}(Y\times \T,\Z); \;\;
&v \otimes \sigma \mapsto (v \wedge e_{r+1}) \otimes \widetilde{\sigma}
\text{ and } \\
\Phi \colon C_{p+1,q+1}^{\BM, \cell}(Y\times \T,\Z) \to C_{p,q}^{\BM, \cell}(Y,\Z);  \;\;
&v \otimes \widetilde{\sigma} \mapsto \pi(\langle w; v \rangle) \otimes \sigma \text{ and }
v \otimes \sigma_\infty \mapsto 0, 
\end{align*}
where  
$\langle \; ; \; \rangle$ denotes the contraction from Definition \ref{def:contract}.

It easy to check $\Phi \circ \Psi = \id$ since $\pi(\langle w; v \wedge e_{r+1} \rangle)=v$.
We define
\begin{align*} 
h \colon C_{p+1,q}^{\BM, \cell}(Y\times \T, \Z) &\to C_{p+1,q+1}^{\BM, \cell}(Y\times \T, \Z); \;\; 
v \otimes \sigma_\infty \mapsto \overline{v} \otimes \widetilde{\sigma} \text{ and } 
v \otimes \widetilde{\sigma}  \mapsto 0,
\end{align*}
where $\overline{v}$ is the map ${\bf F}_p(\sigma) \to {\bf F}_p(\widetilde{\sigma})$ induced by
mapping each vector $v \in {\bf F}_1(\sigma)$ to the unique preimage $\overline{v} \in \pi^{-1}(v)$
with $\langle w ; v \rangle = 0$. 
Then $h$ provides a chain homotopy between
$\id$ and $\Psi \circ \Phi$, thus the lemma is proven. 
\qed
\end{proof}

\begin{cor}\label{prop:pdmatZ}
Let $V$ be a  matroidal fan  in $\R^s$ and set $Y = V \times \T^r$. 
Then $Y$ satisfies Poincar\'e duality. 
\end{cor}

\begin{proof}
  This follows from Proposition \ref{prop:pdmatZ2} and Lemma \ref{lemma product with T}.
  \qed
\end{proof}

\begin{rem} {\rm
 In the following proof we use a local gluing argument. 
To do so, we need to slightly extend our terminology. 
Let $X$ be a polyhedral space with polyhedral structure $\CC$.
Let $U \subset X$ be an open subset. 
A \emph{$\CC$-stratified $q$-simplex} in $U$ is a $\CC$-stratified $q$-simplex $\delta$ in $X$
such that $\delta(\Delta) \subset U$.
Using this convention, Definitions \ref{def:tropchainhomo} and \ref{Borel-Moore} 
can be carried over to  the open set $U$ instead of $X$. 
In particular, we obtain groups $\HH^{p,q}(U, \Z)$ and $\HH^{\BM}_{n-p,n-q}(U, \Z)$. 
Moreover, if $X$ is a tropical space any $\CC$-stratified simplicial structure on $U$ gives rise to a fundamental class
$[U] \in \HH^{\BM}_{n,n}(U, \Z)$ and a map $\cap [U] \colon \HH^{p,q}(U, \Z) \to \HH^{\BM}_{n-p,n-q}(U, \Z)$ 
which do not depend on the simplicial structure. 
Again we say that $U$ satisfies Poincar\'e duality if $\cap [U]$ is an isomorphism for all $p,q$.}
\end{rem}

\begin{proof}[Proof of Theorem \ref{thm:PD}  ]
Let $\CC$ be a polyhedral structure for $X$. 
The proof is completed in two steps. 

\emph{Step 1:
Open stars of faces satisfy Poincar\'e duality:}
A star $U_\sigma$ of a face $\sigma \in \CC$ of a tropical manifold is isomorphic as a tropical manifold 
to a connected neighbourhood $U$ of $(0, (\infty)^r)$ in a polyhedral complex of the 
form $Y = V \times \T^r$, where $V$ is a matroidal fan. 
Note that $\CC$ induces a polyhedral structure $\CC'$ on $Y$.
There is a homeomorphism $f \colon U \to  Y$ 
which preserves the stratification given by $\CC$ and $\CC'$. 
Hence if $\delta$ is a $\CC$-stratified simplex in $U$, the push-forward 
$f_*(\delta) = f \circ \delta$ is a $\CC'$-stratified simplex in $Y$. 
We obtain the following commutative diagram.
\begin{align*}
\begin{xy}
\xymatrix{
\HH^{p,q}(U, \Z)  \ar[rr]^{\cap [U]}  				&& \HH^{\BM}_{n-p,n-q}(U, \Z) \ar[d]^{f_*} \\
\HH^{p,q}(Y, \Z) \ar[rr]^{\cap [Y]}    \ar[u]^{f^*}           && \HH^{\BM}_{n-p,n-q}(Y, \Z)  
}
\end{xy}
\end{align*}
It is straightforward to check that the two vertical arrows are isomorphisms.
Since $\cap [Y]$ is an isomorphism by Corollary \ref{prop:pdmatZ}, 
the map  $\cap [U]$ is also an isomorphism.

\emph{Step 2:
Finite unions of open stars of $\CC$ satisfy Poincar\'e duality:} 
We proceed by induction on the number of open stars in the union with the base case covered above.  
 Suppose that a union of $k$ open stars satisfy Poincar\'e duality. 
Let $U$ be an open star and $V$ be a union of $k$  open stars of $\CC$. 
Then $U \cap V$ is also a union of $k$ open stars,  so that $U$, $V$, and $U \cap V$ satisfy Poincar\'e duality.
The following short exact sequence of complexes (with respect to $\CC$)
\begin{align*}
0 \to C_{p,\bullet}^{\BM, \cell}(U \cup V, \Z) \to C_{p,\bullet}^{\BM, \cell}(U, \Z) \oplus C_{p,\bullet}^{\BM, \cell}(V, \Z) 
\to C_{p,\bullet}^{\BM, \cell}(U \cap V, \Z) \to 0
\end{align*}
induces a Mayer-Vietoris sequence {$\MM_{p,\bullet}^{\BM}(U,V)$} for 
the tropical Borel-Moore homology groups. 
We further denote by $\MM^{p,\bullet}(U, V)$ the Mayer-Vietoris sequence for 
tropical cohomology groups. 
We get a map of sequences
\begin{align*}
\MM^{p,\bullet}(U,V) \to {\MM_{n-p,n-\bullet}^{\BM}(U,V)}
\end{align*}
where in each degree we take the cap product with the appropriate fundamental class. 
Now the claim follows from the five lemma,
since by our assumption $U$, $V$ and $U\cap V$ satisfy Poincar\'e duality. 
\qed
\end{proof}

\begin{cor}\label{pdcor} 
If $X$ is a compact tropical manifold of dimension $n$, then 
\begin{align*}
\HH^{p,q}(X, \Z) \simeq \HH_{n-q,n-q}(X, \Z).
\end{align*}
\end{cor}

\begin{rem}{\rm
A tropical manifold $X$ also satisfies $\HH^{p,q}_c(X, \Z) \cong \HH_{n-p,n-q}(X, \Z)$, 
where $\HH^{p,q}_c(X, \Z) $ denotes cohomology with compact support.
Capping with the fundamental class of a tropical manifold also produces a map 
\begin{align}\label{PDotherway}
 \cap [X] \colon \HH^{p,q}_c(X, \Z) \to \HH_{n-p,n-q}(X, \Z),
\end{align}
which again is an isomorphism for all $p,q$. 
This can be proven by essentially dualising the argument given in this section.}
\end{rem}

The last step in order to prove Theorem \ref{Theorem III intro} is to relate the wave homomorphism on cohomology to its variant on homology. 

\begin{lemma}\label{lemma:wavehomomco}
The following diagram is commutative:
\begin{align*}
\begin{xy}
\xymatrix{
 \HH^{1,1}(X, \Z) \ar[d]^{\cap [X]} \ar[r]^{\phi}   &\HH^{0,2}(X, \R)    \ar[d]^{\cap [X]}    \\
   \HH^{\BM}_{n-1,n-1}(X, \Z)  \ar[r]^{\phihat}  &\HH^{\BM}_{n, n-2}(X, \R) 
 }
\end{xy}
\end{align*}
\end{lemma}

\begin{proof}
  This follows on the level of individual simplices by the definition of $\phi$ (see Definition \ref{def:eigenwave}).
  \qed
\end{proof}

\begin{proof}[Proof of Theorem \ref{Theorem III intro}]
It is easy to check $\phihat \circ \cyc = 0$ (see \cite[Theorem 5.4]{MikZhar}). 
Conversely, let $\beta \in \HH^{\BM}_{n-1,n-1}(X, \Z)$ such that $\phihat(\beta) = 0$. 
Since $X$ satisfies Poincar\'e duality with integral coefficients, 
there exists $\alpha \in  \HH^{1,1}(X, \Z)$ with $\alpha \cap X = \beta$.
By Lemma \ref{lemma:wavehomomco} we have $\phi(\alpha) \cap X = 0$.
Then, again by Poincar\'e duality, we get $\phi(\alpha) = 0$ and 
the statement follows from Theorem \ref{Theorem II intro}.
\qed
\end{proof}

\section{Corollaries and Examples}

In this final section we deduce some corollaries of the main theorems and present some explicit examples. In the case of tropical abelian surfaces and Klein bottles with a tropical structure, we show how to represent $(1,1)$-classes in the kernel of the wave map as fundamental classes of tropical cycles. We also calculate the wave map for two combinatorial types of  smooth tropical quartic  surfaces. We start with some interesting consequences of Theorem \ref{Theorem III intro}. 

\begin{corollary}\label{cor:zerowave}
Let $X$ be a tropical manifold. 
If $\HH^{0,2}(X, \R) = 0$, then every class in $\HH^{\BM}_{n-1,n-1}(X, \Z) $ is the fundamental class of a tropical cycle in $X$.  
\end{corollary}
\begin{proof}
By Poincar\'e duality \ref{thm:PD} we find that $\HH_{n,n-2}^{\BM}(X, \R) = 0$ and 
thus every element of $\HH^{\BM}_{n-1,n-1}(X, \Z)$ is in the kernel of $\phihat$. 
The corollary now follows from Theorem \ref{Theorem III intro}.
\qed
\end{proof}

\begin{corollary}\label{cor:torsionker}
Let $X$ be a tropical manifold. 
If $\alpha \in \HH_{n-1, n-1}^{\BM}(X, \Z)$ is a torsion class, 
then $\alpha$ is the fundamental class of a codimension  one tropical cycle. 
\end{corollary}
\begin{proof}
We have $\phihat(\alpha) = 0 \in \HH^{\BM}_{n,n-2}(X, \R)$. 
Thus the corollary follows again from Theorem \ref{Theorem III intro}.
\qed
\end{proof}

\begin{rem}{\rm
In some instances the image of the  wave homomorphism is contained in
the appropriate  cohomology group with rational instead of real coefficients. 
Then the dimension of the  kernel of the  extension 
$\phi \colon \HH^{1,1}(X, \Z) \otimes \mathbb{Q} \to \HH^{0,2}(X, \mathbb{Q})$ 
gives the rank of the free part of the kernel of 
$\phi \colon \HH^{1, 1}(X, \Z) \to \HH^{0,2}(X, \mathbb{Q})$. 
For example, for the  $\mathbb{Q}$-tropical projective varieties as introduced in  \cite{IKMZ} 
the wave homomorphism is always defined over $\mathbb{Q}$. }
\end{rem}

\subsection{Tropical structures on the Klein bottle} \label{ex:klein}

Let $K$ be a Klein bottle obtained from gluing a parallelogram 
$P \subset \R^2$ with edges $a,b,c,d$ as follows. The edges $b$ and $d$ are glued using the translation in the direction of $a$,
and the edges $a$ and $c$ are glued using an (orientation-reversing) affine transformation $h$ which sends $a$ to $c$ 
with flipped orientation. Let $H \in \text{GL}(2,\R)$ be the linear part of $h$. 
Then the tropical structure given by $\Z^2 \subset \R^2$ extends to $K$ if and only if $H \in \text{GL}(2,\Z)$.
Note that $\det(H) = -1$ and one of its eigenvalues is $-1$, hence the
second eigenvalue is $+1$.  
It follows that the eigenvectors have rational directions and a  computation
shows that the two primitive eigenvectors either form a lattice basis or generate a sublattice of index $2$. 
We can normalise the two cases to the following matrices: 
\begin{align} 
  H_1 &= \begin{pmatrix} -1 & 0 \\ 0 & 1 \end{pmatrix} & H_2 &= \begin{pmatrix} -1 & 1 \\ 0 & 1 \end{pmatrix}
\end{align}
Correspondingly, the parallelogram $P$ has vertices $0$, $l_1 e_1$, $l_1 e_1 + l_2 v_2$, and $l_2 v_2$, where
$v_2$ is either $e_2$ or $\binom{1}{2}$ (see Figures \ref{fig:Klein}, \ref{fig:KleinIndexTwo}). We denote the two Klein bottles
by $K_1$ and $K_2$. 

\begin{figure}[h]
\begin{center}
\begin{tikzpicture}[scale=1.5]
\newcommand{\midarrow}{\tikz \draw[-triangle 90] (0,0) -- +(.1,0);}

\begin{pgfonlayer}{background}
\draw [ultra thick, draw=black] (0,0) rectangle (3,1.5);
\fill[black!20] (0,0) -- +(3,0) -- +(3,1.5) -- +(0,1.5) --cycle;
\end{pgfonlayer}

\begin{pgfonlayer}{foreground}

\draw [ultra thick, draw=blue] (0,0) rectangle (1.5,1.5);
  \fill[blue!20] (0,0) -- +(1.5,0) -- +(1.5,1.5) -- +(0,1.5) --cycle;
    \draw [thick,-latex,blue]  (-.005,.6) -- (-0.005, 0.75) node
  [below right] {$v$}; 
      \draw [thick,-latex,blue]  (1.5,.6) -- (1.5, 0.75) node
  [below right] {$v$}; 
\draw [ thick, draw=red] (0,0) --(3,0);

\draw [thick, draw=blue] (0,0.04) --(1.5,0.04);
  \draw [thick,-latex, red]  (2.25,0) -- (2.25, 0.5) node
  [below right] {$v$}; 
    \draw [thick,->>, red]  (2.25,0) -- (2.5, 0) node
  [below right] {}; 
    \draw [thick,-latex,blue]  (.75,0) -- (.75, 0.5) node
  [below right] {$v$}; 
  \draw [thick,-latex, blue]  (.75,1.5) -- (.75, 2) node
  [below right] {$v$}; 
  
    \draw [thick,->>, blue]  (0,1.2) -- (0, 1.3) node
  [below right] {}; 

    \draw [thick,<<-, blue]  (1.5,1.2) -- (1.5, 1.3) node
  [below right] {}; 
  
      \draw [thick,->>, blue]  (1.2,1.5) -- (1.3, 1.5) node
  [below right] {}; 
      \draw [thick,<<-, blue]  (1.2,0) -- (1.3, 0) node
  [below right] {}; 
  \end{pgfonlayer}
\end{tikzpicture}
\hspace{2cm}
\begin{tikzpicture}[scale=1.5]

\begin{pgfonlayer}{background}
\draw [ultra thick, draw=black] (0,0) rectangle (3,1.5);
\fill[black!20] (0,0) -- +(3,0) -- +(3,1.5) -- +(0,1.5) --cycle;

\end{pgfonlayer}

\begin{pgfonlayer}{foreground}
  
\draw [ ultra thick, draw=red] (0,0) -- (0, 1.5) node [below  right] {};
\draw [ ultra thick, draw=red] (1.5,0) -- (1.5, 1.5) node [below  right] {\textcolor{red}{$-1$}};

  \end{pgfonlayer}
\end{tikzpicture}
\end{center}
\caption{Representing the  torsion class in $\HH_{1, 1}(K_1)$ of a Klein bottle $K_1$ from Subsection \ref{ex:klein} as a parallel class.}

\vspace{-5cm}\phantomsection\label{fig:Klein}

\vspace{5cm}\end{figure}
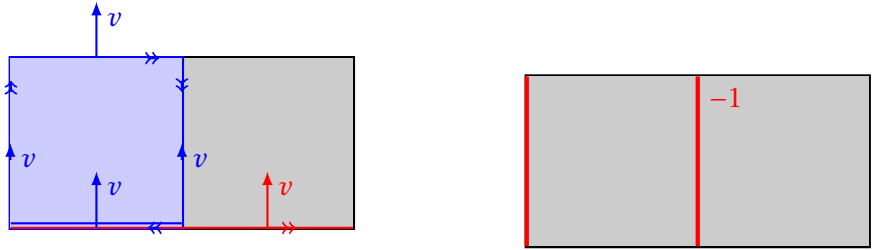

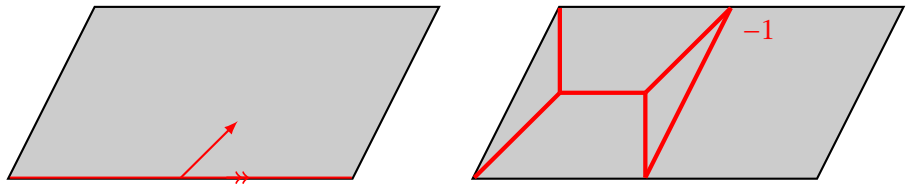
\begin{figure}[h]
\begin{center}

\begin{tikzpicture}[scale=1.5]
\draw [ultra thick, draw=black] (0,0) -- +(.75,1.5) -- +(3.75,1.5) -- +(3,0) --cycle;
\fill[black!20] (0,0) -- +(.75,1.5) -- +(3.75,1.5) -- +(3,0) --cycle;
\draw [ thick, draw=red] (0,0) -- (3,0);
    \draw [thick,->>, red]  (1.9,0) -- (2.1, 0);% node  {}; 
     \draw [thick,-latex,red]  (1.5,0) -- (2, .5); %node {}; 
\end{tikzpicture}
\hspace{0.2cm}
\begin{tikzpicture}[scale=1.5]

\begin{pgfonlayer}{background}
\draw [ultra thick, draw=black] (0,0) -- +(.75,1.5) -- +(3.75,1.5) -- +(3,0) --cycle;
\fill[black!20] (0,0) -- +(.75,1.5) -- +(3.75,1.5) -- +(3,0) --cycle;

%edge in (1,2) direction
\draw [ ultra thick, draw=red] (1.5,0) -- (2.25, 1.5) node [below  right] {\textcolor{red}{$-1$}};

%edge in vertical direction
\draw [ ultra thick, draw=red] (1.5,0) -- (1.5, .75)  {};
\draw [ ultra thick, draw=red] (.75,1.5) -- (.75, .75) {};

%edges in (1,1) direction
\draw [ ultra thick, draw=red] (2.25,1.5) -- (1.5, .75) node [below right] {};
\draw [ ultra thick, draw=red] (0,0) -- (.75, .75) node [below right] {};

%bounded horizontal edge
\draw [ ultra thick, draw=red] (.75,.75) -- (1.5, .75) node [below right] {};
\end{pgfonlayer}

\end{tikzpicture}
\end{center}
\caption{Representing a class in $\HH_{1, 1}(K_2)$ of the Klein bottle $K_2$ 
from Subsection \ref{ex:klein} as a parallel class. }

\vspace{-4cm}\phantomsection\label{fig:KleinIndexTwo}

\vspace{4cm}\end{figure}

Note that $\HH_{2, 0}(K, \Z) = \Z_2$ so that $\HH_{2, 0}(K, \R) = 0$. Hence we are in the situation of Corollary \ref{cor:zerowave}
which says that any $(1,1)$-class can be represented by a tropical cycle of dimension one. 
Let $a = [0, l_1 e_1]$ and $b = [0, l_2 v_2]$ denote two oriented edges of $P$. Any $(1,1)$-class can be represented by
\[
  v \otimes a + \lambda v_2 \otimes b, \; \; \; v \in \Z^2, \lambda \in \Z.
\]
As boundaries of $(1,2)$-chains we obtain
\begin{align} 
  \partial (\binom{x}{y} \otimes P) = \begin{cases}
	  2y e_2 \otimes a & \text{ if } H=H_1, \\
	  y \binom{1}{2} \otimes a & \text{ if } H=H_2.
	\end{cases}
\end{align}
Hence we find $\HH_{1,1}(K_1, \Z) = \Z_2 \oplus \Z \oplus \Z = \langle e_2 \otimes a, e_1 \otimes a, e_2 \otimes b \rangle$ and 
$\HH_{1,1}(K_2, \Z) = \Z \oplus \Z = \langle \binom{1}{1} \otimes a, \binom{1}{2} \otimes  b \rangle$.
Among these generators, the chains $e_1 \otimes a$ and $v_2 \otimes b$ can obviously be represented 
by tropical cycles. Such representations are less obvious for 
the torsion class $e_2 \otimes a$ and the class $\binom{1}{1} \otimes a$.
Explicit representations by tropical cycles are depicted for the two cases in Figures \ref{fig:Klein}, \ref{fig:KleinIndexTwo}. 
Here the chains are drawn in thin red lines with framing and orientation given by simple and double arrows respectively.
The homologous tropical cycles are drawn in thick red lines and labelled with their respective weights if not equal $1$.

For the sake of completeness let us briefly discuss the full classification of tropical Klein bottles.
Instead of just a translation, we may glue the edges $b$ and $d$ via the affine
transformation $x \mapsto Tx + t$, where $t$ is the translation along $a$ and $T \in \text{GL}(2,\Z)$ (before we assumed $T = \id$). 
Depending on $H$, the possible matrices $T$ are of the following two types (see \cite{Sepe2010}) 
\begin{align} 
  T_{1,n} &= \begin{pmatrix} 1 & 0 \\ n & 1 \end{pmatrix}, \ \ \  n \in \Z & T_{2,n} &= \begin{pmatrix} 1 + 2n & -n \\ 4n & 1-2n \end{pmatrix}, \   \ \ n \in \Z. 
\end{align}
The obtained Klein bottles $K_{1,n}$ and $K_{2,n}$ give a full list of Klein bottles with a tropical structure. Analogous to the case $n=0$, we
can compute the homology groups for $n \neq 0$ as
$\HH_{1,1}(K_{1,n}, \Z) = \Z / 2 \Z \oplus \Z / n \Z = \langle e_2 \otimes a, e_2 \otimes b \rangle$ and 
$\HH_{1,1}(K_{2,n}, \Z) = \Z / 2n \Z = \langle \binom{1}{2} \otimes b \rangle$.
Again, it is clear that $v_2 \otimes b$ can be represented by tropical cycles, while 
for $e_2 \otimes a$ the same trick as for $K_1$, Figure \ref{fig:Klein}, is needed.

%previous position of figures 1,2

\subsection{Tropical abelian surfaces}
A tropical abelian surface is $S = \R^2 /\Lambda$ where $\Lambda$ is a rank two lattice 
equal to $ \langle w_1, w_2 \rangle_{\Z}$ for $w_1, w_2 \in \R^2$. 
Therefore $S \cong S^1 \times S^1$.
The sheaf $\mathcal{F}_{\Z}^p$ is the constant sheaf $\bigwedge^p\Z^2$ for $p = 0, 1, 2$, 
and tropical homology groups $\HH_{p, q}(S, \Z)$ are free $\Z$ modules whose ranks are given by the follow tropical Hodge diamond, 
$$\begin{array}{ccccc} &  & 1 &  &  \\ & 2 &  & 2 &  \\1 &  & 4 &  & 1 \\  & 2 &  & 2 &   \\  &  & 1 &  & \end{array}.$$
We can choose a basis of $\HH_{1, 1}(S, \Z)$ as $\alpha_{ij} = e_i \otimes \sigma_j$ 
where $\sigma_1, \sigma_2$ form a basis of $\HH_{1}(S, \Z)$ and $e_1, e_2$ are a lattice basis of $\Z^2$.  
Furthermore suppose that $\sigma_i$ is the quotient of the oriented line in $\R^2$ in direction $w_i$. 
Then the eigenwave homomorphism is given by $\hat{\phi} (\alpha_{ij}) = e_i \wedge w_j$. 

We can explicitly describe a parallel representative of $\alpha \in \HH_{1, 1}(S, \Z) \cap \ker(\hat \phi)$. 
For $\alpha \in \HH_{1, 1}(S, \Z)$ we can write $\alpha = v_1 \otimes \sigma_1 + v_2 \otimes \sigma_2$, 
where $v_1$ and $v_2$ are integer vectors. 
Then $\alpha $ is in $\ker(\hat \phi)$ if and only if $v_1 \wedge w_1 + v_2 \wedge w_2 = 0$. 

Suppose that $v_1$ and $v_2$ are linearly independent. 
Consider the triangle  $T$ in $\R^2$ with vertices $0, w_1,$ and $w_2$. 
Firstly, we claim that if $v_1 \wedge w_1 + v_2 \wedge w_2 = 0$, 
then the lines in directions $v_1 + v_2$, $v_1$, and $v_2$ drawn from the vertices 
$0, w_1$, and $w_2$, respectively, are concurrent. 

Since $v_1 \wedge v_2 \neq 0$, it forms  a $\R$-basis of $\bigwedge^2 \R^2$
and hence there exists an $\alpha \in \R$ such that $\alpha v_1 \wedge v_2 = v_1 \wedge w_1 = - v_2 \wedge w_2$. 
Then the  three lines mentioned above intersect at the point $p = \alpha ( v_1 + v_2)$ in $\R^2$. 
To see this notice that $ (x - w_i) \wedge v_i = 0$ is the defining equation for the line from $w_i$. 
Then 
\begin{align*}
  (\alpha (v_1 + v_2)  - w_i) \wedge v_i = \alpha(v_j \wedge v_i) + v_i \wedge w_i \quad \text{for} \quad i \neq j.
\end{align*}
Consider the  $(1,1)$-cycle $\alpha'  = (v_1 + v_2) \otimes [0,p] - v_1 \otimes [w_1, p] - v_2 \otimes [w_2,p]$.
Then $\alpha - \alpha' = \partial(\beta_1 + \beta_2)$, where 
$\beta_i =  v_i \otimes \tau_i$ is a $(1, 2)$-simplex based on the triangle $\tau_i = [0 w_i p]$
(with given orientation).
Then  $\alpha'$ is the fundamental class of a tropical $1$-cycle and it is homologous to $\alpha$, see Figure \ref{fig:abelian}.

If $v_1$ and $v_2$ are linearly dependent and $\alpha \in \ker(\hat \phi)$ then $w_1 + w_2 = \alpha v_1$ for some $\alpha \in \R$. 
In particular, $w_1 + w_2$ is a rational direction and 
$\alpha$ can be represented by a parallel cycle supported on the diagonal of a  fundamental domain for $S$.

Let us make a few remarks about general $n$-dimensional tropical tori $S$. A $(1,1)$-class $\gamma = \sum_{i,j} \gamma_{ij} e_i^* \otimes e_j^*
\in \HH^{1, 1}(S, \Z)$ lies in the kernel of $\phihat$ if and only if $(\gamma_{ij})$ is symmetric. We do not know an explicit
construction of a tropical divisor $D \in Z_{n-1}(S)$ such that $[D] = \gamma \cap [S]$. 
However, if in addition $(\gamma_{ij})$ is positive definite (in this case, $\gamma$ is called a polarization of $S$), 
an explicit representative of $\gamma \cap [S]$ is given by the theta divisor $\Theta_\gamma$ associated to $\gamma$, 
see \cite{MikZha:Theta}. 
This follows from \cite[Lemma 5.2]{MikZha:Theta}, Theorem \ref{thm:commutingdiagram}
and the explicit description of $c_1$ used in Proposition \ref{boundary and wave}.
Note that in this case, $\gamma \cap [S]$ can be represented by an effective divisor. 
Moreover, if a polarization $\gamma$ exists, 
representatives for other classes $\gamma'$ can be constructed
as $\Theta_{\gamma' + n \gamma} - n \Theta_\gamma$ for $n$ sufficiently large. 
It would be interesting to find general criteria for effective representability and to describe 
explicit constructions of effective representatives in such cases.

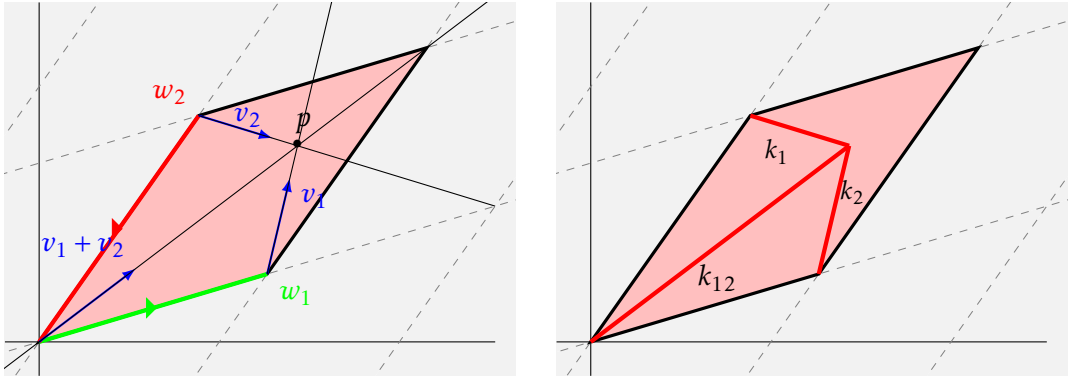
\begin{figure}
\begin{center}
\begin{tikzpicture}[scale=1.5]

\coordinate (Origin)   at (0,0);
\coordinate (XAxisMin) at (-0.3,0);
\coordinate (XAxisMax) at (4,0);
\coordinate (YAxisMin) at (0,-0.3);
\coordinate (YAxisMax) at (0,3);

\begin{pgfonlayer}{foreground}
  \draw [thin] (XAxisMin) -- (XAxisMax);% Draw x axis
  \draw [thin] (YAxisMin) -- (YAxisMax);% Draw y axis
\end{pgfonlayer}

\clip (-0.3,-0.3) rectangle (4.18cm,3cm); % Clips the picture...
\begin{pgfonlayer}{background}
\clip (-0.3,-0.3) rectangle (4.18cm,3cm); % Clips the picture...
\end{pgfonlayer}
\begin{pgfonlayer}{foreground}
\clip (-0.3,-0.3) rectangle (4.18cm,3cm); 

\end{pgfonlayer}

\pgftransformcm{1}{0.3}{0.7}{1}{\pgfpoint{0cm}{0cm}}

\newcommand{\midarrow}{\tikz \draw[-triangle 90] (0,0) -- +(.1,0);}

\coordinate (Bone) at (0,2);
\coordinate (Btwo) at (2,-2);
\coordinate (negBonePlusBtwo) at (-2,0);
\coordinate (vone) at (1.5, 1);
\coordinate (vtwo)  at (1,1.5 );
\coordinate (vonetwo) at (0.5 , 0.5);

\begin{pgfonlayer}{foreground}
  \draw[style=help lines,dashed] (-14,-14) grid[step=2cm] (14,14);
\end{pgfonlayer}

\begin{pgfonlayer}{background}
\foreach \x in {-2,-1,...,3}{
  \foreach \y in {-2,-1,...,4}{
  \pgfmathparse{int(3*(\x+\y+4))}
   \fill[black!5] (2*\x,2*\y) -- +(2,0) -- +(2,2) -- +(0,2) --cycle;

  }  
}

\fill[red!25] (0,0) -- +(2,0) -- +(2,2) -- +(0,2) --cycle;
\end{pgfonlayer}

\begin{pgfonlayer}{foreground}
  \draw[very thick,draw=black] (Origin) rectangle ($2*(Bone)+(Btwo)$);

  \draw [ultra thick, green] (Origin) -- node {\midarrow} ($(Bone)+(Btwo)$) node
  [below right] {$w_1$};
  \draw [ultra thick, red] (Origin) -- node {\midarrow} (Bone) node [above left]
  {$w_2$};
  
  \draw [thick,-latex, blue]  ($(Bone)+(Btwo)$) -- (vone) node
  [below right] {$v_1$}; 
  \draw [thick,-latex, blue] (Bone)--(vtwo) node [above left]
  {$v_2$};  
  \draw [thick,-latex, blue] (Origin)--(vonetwo) node [above left]
  {$v_1+v_2$};

   \node at (1.3333333,1.3333333) {\textbullet};
      \node [above ] at (1.39,1.3333333) {$p$};              
    \draw [] (-1,-1) -- (4,4);
    \draw [] (0,2 )--(4,0  );
       \draw [] (2,0 )--(0,4  );
\end{pgfonlayer}
\end{tikzpicture}
\hspace{0.3cm}
\begin{tikzpicture}[scale=1.5]
\coordinate (Origin)   at (0,0);
\coordinate (XAxisMin) at (-0.3,0);
\coordinate (XAxisMax) at (4,0);
\coordinate (YAxisMin) at (0,-0.3);
\coordinate (YAxisMax) at (0,3);

\newcommand{\midarrow}{\tikz \draw[-triangle 90] (0,0) -- +(.1,0);}
\begin{pgfonlayer}{foreground}
  \draw [thin] (XAxisMin) -- (XAxisMax);% Draw x axis
  \draw [thin] (YAxisMin) -- (YAxisMax);% Draw y axis
\end{pgfonlayer}

\clip (-0.3,-0.3) rectangle (4.18cm,3cm); % Clips the picture...
\begin{pgfonlayer}{background}
\clip (-0.3,-0.3) rectangle (4.18cm,3cm); % Clips the picture...
\end{pgfonlayer}
\begin{pgfonlayer}{foreground}
\clip (-0.3,-0.3) rectangle (4.18cm,3cm); % Clips the picture...

%\draw[blue] (-0.25,0.25) node {$\mathbf{A}$};
%\draw[blue] (0.35,0.25) node {$\mathbf{B}$};
%\draw[blue] (-0.3,-0.25) node {$\mathbf{C}$};   
%\draw[blue] (0.25,-0.25) node {$\mathbf{D}$};
\end{pgfonlayer}

\pgftransformcm{1}{0.3}{0.7}{1}{\pgfpoint{0cm}{0cm}}

\coordinate (Bone) at (0,2);
\coordinate (Btwo) at (2,-2);
\coordinate (negBonePlusBtwo) at (-2,0);
\coordinate (vone) at (1.5, 1);
\coordinate (vtwo)  at (1,1.5 );
\coordinate (vonetwo) at (1.5 -1 , -1+3.5);

\begin{pgfonlayer}{foreground}
  \draw[style=help lines,dashed] (-14,-14) grid[step=2cm] (14,14);
\end{pgfonlayer}

\begin{pgfonlayer}{background}
\foreach \x in {-2,-1,...,3}{
  \foreach \y in {-2,-1,...,4}{
  \pgfmathparse{int(3*(\x+\y+4))}
 % \fill[red!\pgfmathresult] (2*\x,2*\y) -- +(2,0) -- +(2,2) -- +(0,2) --cycle;
   \fill[black!5] (2*\x,2*\y) -- +(2,0) -- +(2,2) -- +(0,2) --cycle;

  }  
}

\fill[red!25] (0,0) -- +(2,0) -- +(2,2) -- +(0,2) --cycle;
\end{pgfonlayer}

%\foreach \x in {-1,0,...,4}{
 % \foreach \y in {-1,0,...,4}{
    %    \node[star,scale=5,draw=black,fill=white] at (2 * \x, 2 * \y) {};
        % make a lattice of stars
  %}
%}

\begin{pgfonlayer}{foreground}
  \draw[very thick,draw=black] (Origin) rectangle ($2*(Bone)+(Btwo)$);

%  \draw[blue] (Bone) node [below right] {$\mathbf{D} '$}; \draw[blue]
  %($(Bone)+(Btwo)$) node [above left] {$\mathbf{w_1} '$}; 
  %\draw[blue](Origin) node [below left] {$\mathbf{0} $};

%  \draw [ultra thick] (Origin) -- ($(Bone)+(Btwo)$) node
  %[below right] {$w_1$};
  %\draw [ultra thick, red] (Origin) -- node {\midarrow}  (Bone) node [above left]
  %{$w_2$};
  
  %
%  \draw [thick,-latex, blue]  ($(Bone)+(Btwo)$) -- (vone) node
 % [below right] {$v_1$}; 
 % \draw [thick,-latex, blue] (Bone)--(vtwo) node [above left]
 % {$v_2$};  
%  \draw [thick,-latex, blue] (Origin)--(vonetwo) node [above left]
  %{$v_2$};
  
%  \draw [thick,dashed,-latex,red,scale=3] (Origin) -- ($-0.25*(Bone)$)
  %; \draw [thick,dashed,-latex,red] (Origin) --
%  ($0.35*(negBonePlusBtwo)$) ;

  \node at (1.7333333,.803333333) {\small{$k_2$}};     
  
  \node [above right] at (.5,1.3333333) {\small{$k_1$}};     
  \node at (.9,.3) {$k_{12}$};     
%  \node [above right] at (1.3333333,1.3333333) {\small{$\alpha$}};          
    \draw [red, ultra thick] (0,0) --(1.3333333,1.3333333) ;
    \draw [red, ultra thick] (0,2 )--(1.3333333,1.3333333) ;
       \draw [red, ultra thick] (2,0 )--  (1.3333333,1.3333333) ;
\end{pgfonlayer}
\end{tikzpicture}
\end{center}
\caption{A $(1,1)$-cycle on a tropical abelian surface in red and green on the left, 
and a representative of its class by the fundamental class of a tropical $1$-cycle.}

\vspace{-7cm}\phantomsection\label{fig:abelian}

\vspace{7cm}
\end{figure}

\subsection{Tropical hypersurfaces}

A tropical hypersurface in $\R^{n+1}$ or in an $n+1$-dimensional tropical toric variety is the divisor (as in Definition \ref{def:divcyc}) 
of a tropical polynomial function.
It is an $n$-dimensional polyhedral complex which is dual to a 
regular subdivision of a lattice polytope. This implies the following statement.

\begin{prop}\label{prop:bouquet}
A tropical hypersurface is homotopic to a bouquet of spheres. 
\end{prop}

A tropical hypersurface is \emph{non-singular} if it is  a tropical manifold. 
This is the case  if and only if it is dual 
to a  regular subdivision of a lattice polytope which is \emph{primitive}, 
i.e.~if every top dimensional polytope in the subdivision is of lattice volume $1$.

\begin{corollary}
If $X$ is a non-singular tropical hypersurface in an $n+1$-dimensional tropical toric variety for $n \geq 3$ 
then every class in $\HH^{\BM}_{n-1, n-1}(X, \Z)$ is the fundamental class of a tropical cycle in $X$. 

\end{corollary}

\begin{proof}
It follows from  Proposition \ref{prop:bouquet} 
that 
$\HH^{0,2}(X, \R)  = 0$, 
so the statement follows from Corollary \ref{cor:zerowave}.
\qed
\end{proof}

The first example of non-trivial wave maps for tropical hypersurfaces is in the case of  smooth tropical quartic  surfaces.
We now look at two specific  examples.

\begin{defn}
A smooth  tropical  quartic surface is dual to a primitive regular triangulation of a size $4$ tetrahedron.
\end{defn}

There is one $2$-dimensional polytopal sphere $P$ contained in smooth tropical quartic  surface. 
It is dual to all cells of the regular subdivision of the size $4$  tetrahedron  
which contain the unique interior lattice point of the size $4$ simplex of dimension $3$. 
The Betti diamond of a  smooth tropical quartic surface is:
$$\begin{array}{ccccc} &  & 1 &  &  \\ & 0 &  & 0 &  \\1 &  & 20 &  & 1 \\  & 0 &  & 0 &   \\  &  & 1 &  & \end{array}$$
so the wave map sends a $\Z$-module of  rank $20$ to a $1$-dimensional real vector space. 
The \emph{Picard rank} of smooth tropical quartic surface $X$ is the rank of $\Pic(X)$. 
Since $\HH^{0,1}(X, \Z) = 0$ for a smooth tropical quartic  surface $X$ the map $c_1 \colon \Pic(X) \to \HH^{1, 1}(X, \Z)$ is injective by the long exact sequence obtained from (\ref{exponential sequence}). Hence the Picard rank is equal to the rank of the kernel of the wave homomorphism by Proposition \ref{boundary and wave}.

\begin{example}[A tropical surface with Picard rank $19$]\label{ex:pic19}{\rm
A tropical hypersurface $X$ with Newton polytope $n+1$-simplex of size $d$  is floor decomposed
if the relative interior of every top dimensional polytope in the  dual subdivision of $X$ 
does not intersect the hyperplanes
$\{x_{n+1} = i\}$ for all $0\leq i \leq d-1\}$. 
For examples see \cite{Shaw:11Homol}. 

A floor decomposed tropical quartic  surface is determined (up to choice of constants regulating the height of the floors) 
by a collection of  non-singular planar tropical curves $C_1, C_2, C_3, C_4$ 
where each $C_d$ is of degree $d$ (i.e.~dual to a primitive triangulation of the size $d$ lattice triangle). 

\begin{figure}
\begin{center}
\includegraphics[scale = 0.4]{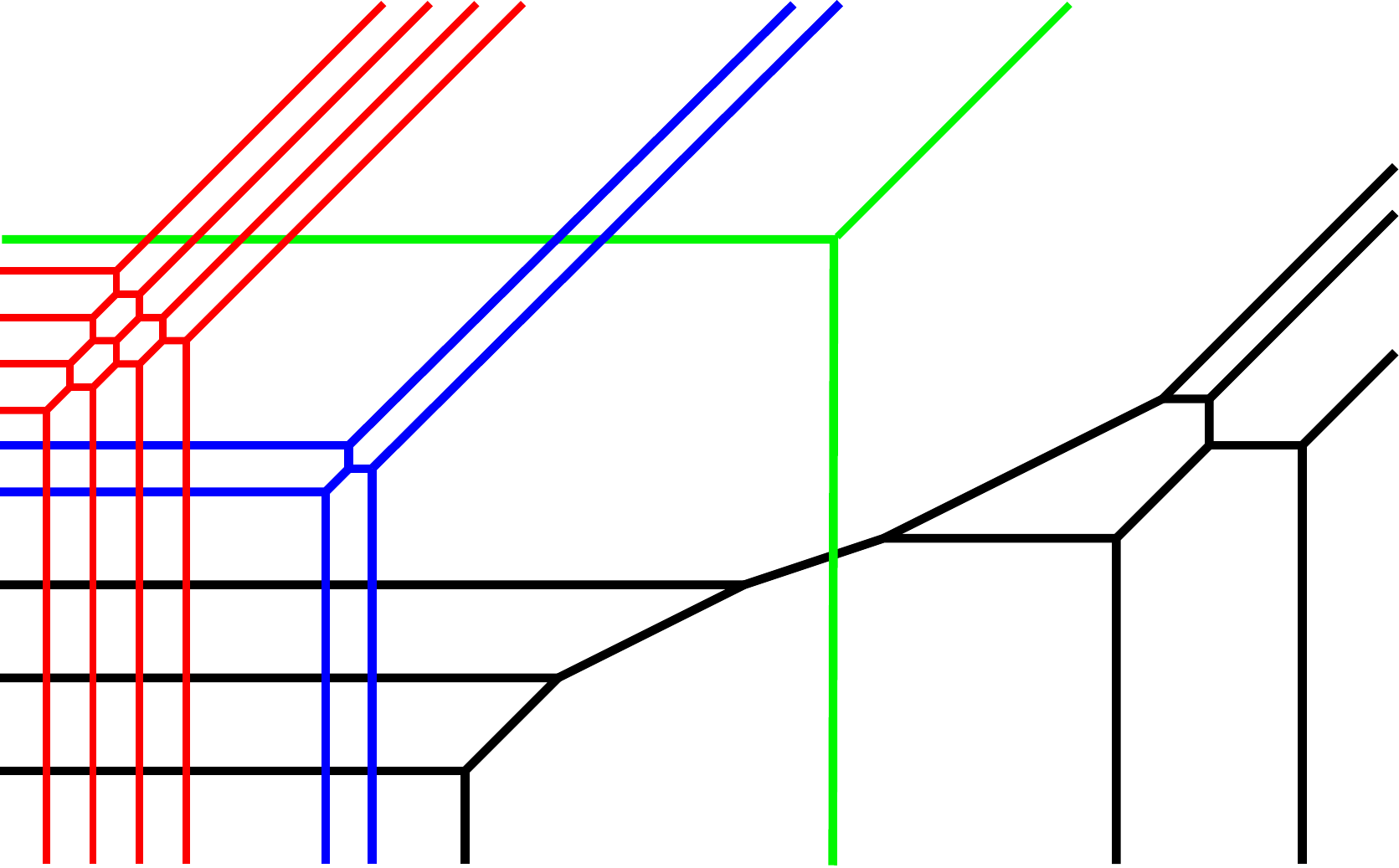}
\end{center}
\caption{The floor diagram of a smooth tropical quartic  surface with  $\text{rank}(\HH^{1, 1}(X, \Z) \cap \ker(\phi)) \geq 18$. }

\vspace{-6cm}
\phantomsection\label{fig:floor}

\vspace{6cm}
\end{figure}

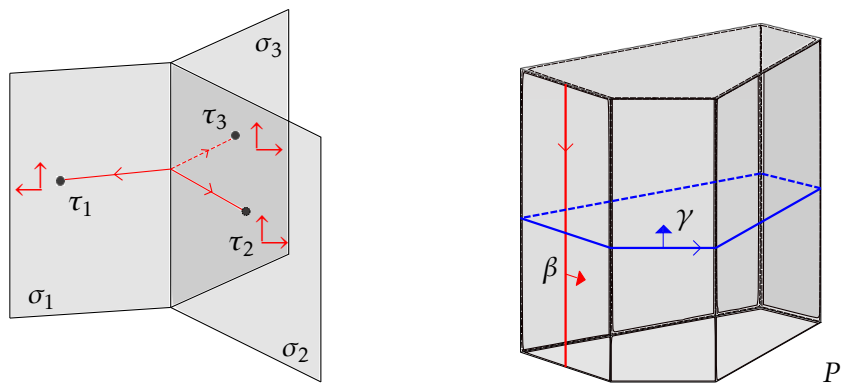
\begin{figure}
\begin{center}
\includegraphics[scale=0.5]{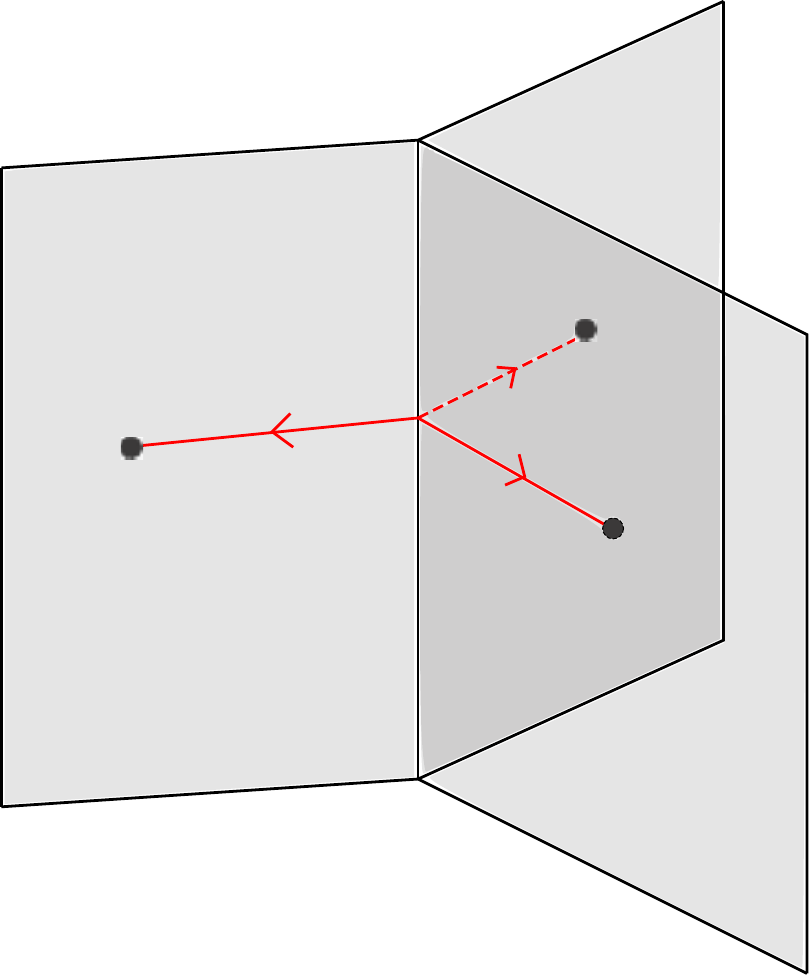}
\put(-95,65){$\tau_1$}
\put(-35,50){$\tau_2$}
\put(-45,98){$\tau_3$}
\put(-110,30){$\sigma_1$}
\put(-15,10){$\sigma_2$}
\put(-25,125){$\sigma_3$}
\put(-108,75){\textcolor{red}{$\uparrow$}}
\put(-115,70){\textcolor{red}{$\leftarrow$}}
\put(-25,55){\textcolor{red}{$\uparrow$}}
\put(-23,50){\textcolor{red}{$\rightarrow$}}
\put(-27,90){\textcolor{red}{$\uparrow$}}
\put(-25,85){\textcolor{red}{$\rightarrow$}}
\hspace{2.5cm}
\includegraphics[scale=0.35]{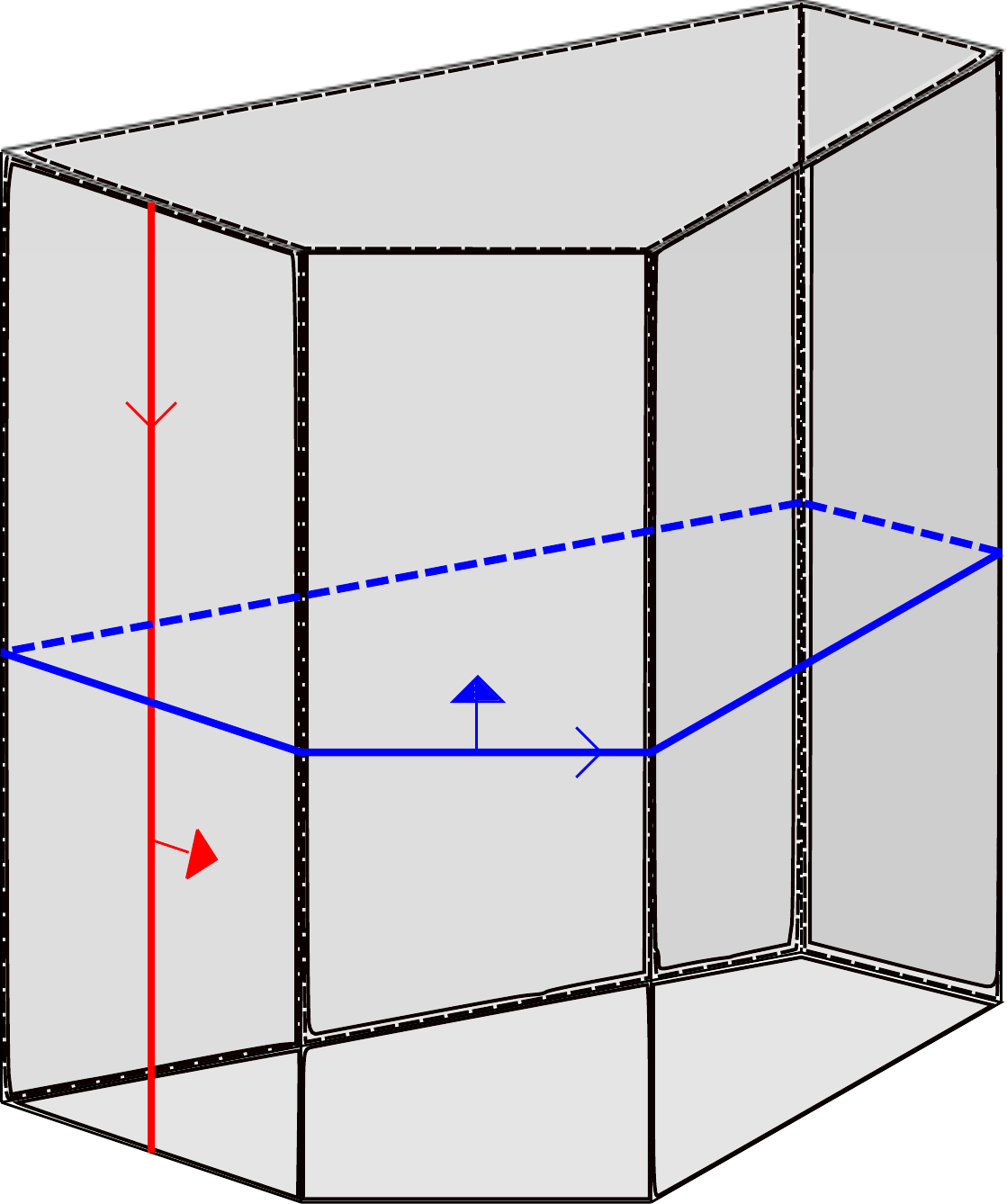}
\put(0, 0){$P$}
\put(-55, 60){$\gamma$}
\put(-105, 40){$\beta$}
\end{center}
\caption{On the left the branched path which is the support of a $(2, 1)$-cycle whose boundary is $\tau_1 + \tau_2 + \tau_3$. 
On the right a depiction of the polytope $P$ and  the cycles $\beta$ and $\gamma$ from Example \ref{ex:pic19}. }

\vspace{-7cm} \phantomsection   \label{fig:pentaprism}

\vspace{7cm}
\end{figure}

Given a floor decomposed surface $X$  a basis of its $\HH_{1, 1}(X, \Z)$ tropical homology 
was described in \cite{Shaw:11Homol}. 
On a floor given by the curves $C_{i}$ and $C_{i+1}$, there are $i(i+1) -1$
independent ``floor cycles". 
This produces $11 + 5 + 1 = 17$ independent $(1, 1)$-cycles. 
They can be chosen such that their support, after projecting to the plane,
forms a minimal loop in $C_{i} \cup C_{i+1}$ not contained in $C_{i}$ or $C_{i+1}$.
By our particular choice of curves $C_1, \dots, C_4$, any such cycle is disjoint from the cycle of $C_3$.
Hence we can assume the floor cycles do not intersect the polytopal sphere. 
Additionally, there is a cycle $h$, a multiple of which is the hyperplane section, 
which can also be made disjoint from $P$.
Together with the cycles $\alpha$ and $\beta$ which are illustrated on the right hand side of Figure \ref{fig:pentaprism}  and 
completely described in \cite{Shaw:11Homol}, these cycles form a basis of $\HH_{1,1}(X, \Z)$.

We now describe how $(2,0)$ cells behave when passing to homology depending on their supporting point.
Orient each face of $P$ so that it is the boundary of the $3$-dimensional polytope. 
At any edge $\gamma$ of $X$ of sedentarity $\emptyset$ there are three faces adjacent to it: 
$\sigma_1, \sigma_2, $ and $\sigma_3$. 
If $\tau_i $ is an appropriately oriented generator of $\mathbf{F}_2^\Z(\sigma_i)$ 
we can find a $(2, 1)$-cell whose support is the branched path on the left of Figure \ref{fig:pentaprism} and 
whose boundary is  $\tau_1 + \tau_2 +  \tau_3$.  
Moreover, for  any point $x$ not on the polytopal sphere we can find a branched path in $X$ 
whose endpoints are $x$ and points of positive sedentarity, thus showing that a $(2, 0)$-cycle supported on $x$ is homologous to $0$. 
In addition, for two faces $\sigma_1$, $\sigma_2$ of the polytopal sphere, we have $\tau_1 \sim \tau_2$ 
where $\tau_i$ are appropriately oriented generators of $\mathbf{F}_2^\Z(\sigma_i)$ and 
such an $\tau_i$ generates $\HH_{2,0}(X, \Z)$.  
We denote the class of these $\tau_i$ by $\tau$.

This implies that $\hat{\phi}(h) = \hat{\phi}(\gamma) =0$ for all floor cycles $\gamma$. 
Moreover, using this description of $\HH_{2,0}(X, \R)$ we can explicitly compute
$\hat{\phi}(\alpha) = l \tau$ and $\hat{\phi}(\alpha) = h \tau$, where
$l$ is the lattice length of the unique cycle in $C_3$
(i.e., the $j$-invariant of $C_3$)
and $k$ is the lattice height 
of the pentagonal prism $P$
(or, the distance between the floors connected by $C_3$). 
Varying the coefficients of the defining tropical polynomial, these two parameters can be
controlled independently. In particular, we can arrange both
$\frac{l}{k} \in \mathbb{Q}$ and $\frac{l}{k} \notin \mathbb{Q}$.

We conclude that $\Ker(\phihat)$ has rank $19$ or $18$, depending on our choice and thus 
any tropical quartic surface with the combinatorial type of the one chosen above must have Picard rank equal to $18$ or $19$
by Theorem \ref{Theorem III intro}.}
\end{example}

\begin{example}[A smooth tropical quartic surface with Picard rank 1]\label{ex:pic1}{\rm
The second example is dual to a cone triangulation. 
Fix a primitive regular triangulation of each of the four two dimensional faces of the size $4$ tetrahedron. 
We obtain a unique primitive regular triangulation by considering the cone over this triangulation 
with the cone point being the unique interior lattice point $(1, 1, 1)$ of the size $4$ simplex. 
See Figure \ref{fig:conek3} for an example. 

\begin{figure}
\begin{center}
\includegraphics[scale=0.3]{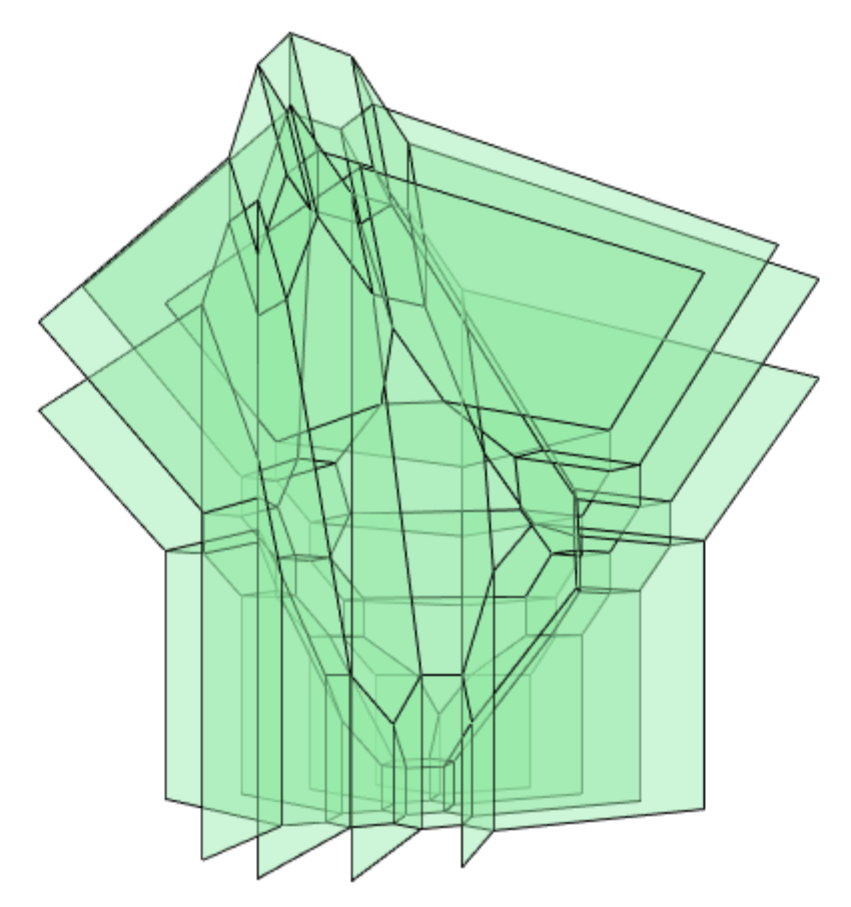}\hspace{1cm}
\includegraphics[scale=0.3]{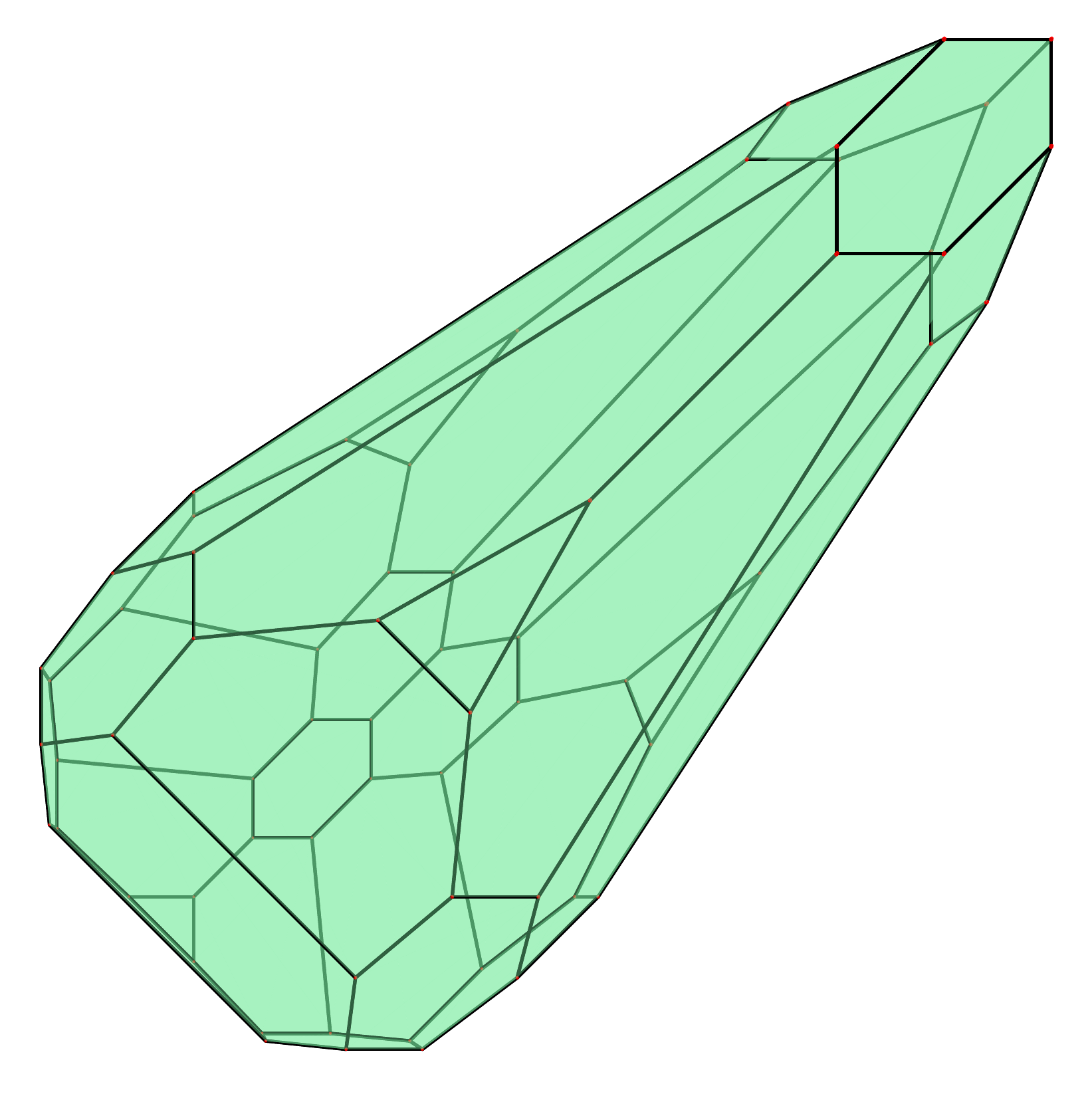}
\end{center}
\caption{A  tropical quartic  surface and its bounded polytope from Example \ref{ex:pic1}.
}

\vspace{-6.5cm}
\phantomsection
\label{fig:conek3}

\vspace{6.5cm}
\end{figure}

In this case all $34$ bounded $2$-dimensional faces of $X$ are faces of the polytopal sphere $P$ contained in $X$. 
Each such face corresponds to a unique lattice point on the boundary of the tetrahedron. 
There are $3$ types of such points: the $4$ vertices of the tetrahedron, 
the $3 \times 6 = 18$ lattice points on the edges of the tetrahedron, 
and the  $3\times 4 = 12$ lattice points contained in the relative interior a  $2$-dimensional  face of the tetrahedron. 
Any $(2,0)$-cycle whose support is not contained on the polytopal sphere $P$ is homologous to zero 
since its support is then contained in unbounded faces. 
Orient each two dimensional face of $P$ so that the collection of faces form the boundary of the bounded $3$-dimensional polytope 
in the complement of $X$ in $\R^3$. 
As in the previous example, equipping any $p \in \sigma \subset P$ with the unique  generator of
$\mathbf{F}_2^\ZZ(\sigma)$ oriented coherently with respect to $\sigma$, 
we obtain representatives of the (same) generator $\tau \in \HH_{2,0}(X, \Z)$.

Each  bounded $2$-dimensional face $\sigma$ 
provides a $(1,1)$-cycle by taking its boundary and equipping each with a coefficient in  $\BF_1^{\ZZ}$
which is a vector generating $\Z^3$ together with the lattice parallel to  $\sigma$.
 Each bounded two dimensional face of $X$ corresponds to a lattice point on the boundary of $\Delta$. 
We denote such a cycle by $\alpha_a$ where $a$ is the corresponding lattice point in $\partial \Delta$. 
If $a$ is in the interior of a two dimensional face of $\partial \Delta$, 
then $\alpha_{a} = 0$ in homology. This leaves $22$ such $(1,1)$-cycles.
Suppose the defining polynomial of $X$ is $f(x) = ``\sum_{a \in \Delta} c_a x^a"$.
Up  to sign we have
$\phihat(\alpha_a)  = w_{a} \tau $ where, 
\begin{itemize}
\item[\rm (1)] if  $a$ in the relative interior of an edge of $\Delta$ with primitive integer direction $v$, then 
\begin{align*}
w_a = c_{a +v} - c_{a-v}; 
\end{align*}
\item[\rm (2)] if $a$ is a vertex of $\Delta$, let $a_1, a_2, a_3$ denote the three lattice points 
in the relative interiors of edges of $\Delta$ which are of lattice distance one  away from $a$. Then 
\begin{align*}
w_{a} = c_{a_1} + c_{a_2} + c_{a_3} - 3c_{(1,1,1)}.
\end{align*}
\end{itemize}
Let $W \subset \HH_{1, 1}(X, \QQ)$ denote the subspace spanned by the $22$ cycles.
It turns out that $\dim(W) = 19$, and that $W$, together with the hyperplane section $h$, generate $\HH_{1, 1}(X, \QQ)$. 
By choosing a basis for $W$ among the $\alpha_a$'s, we can identify $\Hom_{\QQ}(W, \R)$ with $\R^{19}$. 
Let $V_\text{coef} = \R^{\Delta \cap \Z^3}$ be the vector space of
polynomials $f = (c_a)_a$ and let $C \subset V_\text{coef}$ denote the cone 
of coefficients of tropical polynomials of the fixed  combinatorial type.
Then $\phihat$ induces a linear map  $w \colon V_\text{coef} \to \R^{19}$ 
which is explicitly given by formulas (1) and (2). It can be checked that $s$ has full rank and hence $w(C) \subset \R^{19}$
has non-empty interior. 
For any $1 \leq r \leq 19$, let $Y_r \subset \R^{19}$ be the subset of vectors
whose entries span a $20-r$-dimensional $\QQ$-subspace of $\R$. 
Since $Y_r$ is dense for any $r$, there exists a tropical polynomial  $f \in C$ with 
$w(f) \in Y_r$. Such an $f$ describes a tropical surface $X$ with Picard rank equal to $r$. 
In particular, we can produce a tropical surface of Picard rank equal to $1$.
This proves Theorem \ref{Prop Intro} from the introduction. }
\end{example}

\providecommand{\bysame}{\leavevmode\hbox to3em{\hrulefill}\thinspace}
%\providecommand{\MR}{\relax\ifhmode\unskip\space\fi MR }
%\providecommand{\MRhref}[2]{%
%  \href{http://www.ams.org/mathscinet-getitem?mr=#1}{#2}
%}
%\providecommand{\href}[2]{#2}
%\begin{thebibliography}{{Kem}92}
%
%
%\end{thebibliography}

\bibliographystyle{amsalpha}
\bibliographymark{References}
% This would change the heading "References" by "Bibliography"
% \renewcommand{\refname}{Bibliography}
\def\cprime{$'$}

\end{document}